\documentclass[11pt,a4paper]{article}

\usepackage{latexsym}
\usepackage{psfrag}
\usepackage{amsmath,amssymb,cite}
\usepackage{latexsym}
\usepackage{graphicx}
\usepackage{lscape}
\usepackage{afterpage}
\usepackage[color=green!40]{todonotes}
\usepackage[normalem]{ulem}
\usepackage{dutchcal}

\usepackage{hyperref}
\hypersetup{
    colorlinks=true,
    linkcolor=blue,
    filecolor=magenta,      
    urlcolor=blue,
    citecolor=red,
}
\DeclareMathOperator*{\card}{card}

\newcommand{\ds}{\displaystyle}
\newcommand{\nexto}{\kern -0.54em}
\newcommand{\dR}{{\rm {I\ \nexto R}}}

\newcommand{\dZ}{{\cal Z \kern -0.7em Z}}
\newcommand{\dC}{{\rm\hbox{C \kern-0.8em\raise0.2ex\hbox{\vrule
height5.4pt width0.7pt}}}}
\newcommand{\dQ}{{\rm\hbox{Q \kern-0.85em\raise0.25ex\hbox{\vrule
height5.4pt width0.7pt}}}}
\newcommand{\proofbox}{\hspace{\fill}{$\Box$}}
\newtheorem{lemma}{Lemma}
\newtheorem{theorem}{Theorem}
\newtheorem{corollary}{Corollary}

\newtheorem{remark}{Remark}

\newenvironment{proof}{Proof.}{\proofbox}

\begin{document}

\author{Authors}

\author{Markus Hegland\footnote{Mathematical Sciences Institute, The Australian National University, Acton, ACT 2601, 
  Australia. E-mail: markus.hegland@anu.edu.au}
\and
C. Yal{\c c}{\i}n Kaya\footnote{Mathematics, UniSA STEM, University of South Australia, Mawson Lakes, S.A. 5095, 
  Australia. E-mail: yalcin.kaya@unisa.edu.au\,.}
}

\title{\vspace{-10mm}\bf Probability Density Estimation via Optimal Control}

\date{1 October 2025}

\maketitle

\thispagestyle{empty}

\begin{abstract}
{\sf We employ optimal control theory to study the problem of estimating the probability density function from a data set originating from an unknown probability distribution.  The original variational problem is reformulated as a multi-stage optimal control problem and the associated maximum principle, or conditions of optimality, is reduced to a two-point boundary-value problem with interior conditions.  A numerical scheme is proposed to solve the discretization of this problem.  Estimates of density functions for synthetic and real data are computed using the proposed approach.  The real data come from the Old Faithful geyser and the speeds of a group of galaxies. Comparisons are made with the popular statistics software R.
}
\end{abstract}

\begin{verse}
\end{verse}

\begin{verse}
\end{verse}

\pagestyle{myheadings}
\markboth{}{\sf\scriptsize Density Estimation via Optimal Control\ \ by M.~Hegland and C.~Y.~Kaya}

\section{Introduction}

Suppose that $t_i$, $i=1,\ldots,n$, are points sampled from an unknown probability distribution.  The process of finding an approximation of the probability density function of the distribution, from which these sample data points come, is referred to as {\em density estimation}.  For a comprehensive review of existing methods, in particular for {\em nonparametric density estimation}, which is the focus of interest of this paper, we refer the reader to~\cite{GriHeg2010, Shvartsman2010, Silverman1986}.  In what follows, we provide a brief review of some of these approaches to furnish our context.

\subsection{Context and relevance}

Obviously, the simplest density estimation is to construct a {\em histogram} of the set of data points, $\{t_1,\ldots,t_n\}$.  Histograms have been widely used as a visual representation of the distribution of quantitative data since the 18th century~\cite{Ioannidis2003}.  However, through a histogram, one can only get a discrete, and usually poor, approximation of the underlying probability density function (after normalizing the total area of the ``rectangles'' in the histogram to 1, of course).  Any continuous function fitted, for example to points chosen on the upper edges of the rectangles in a histogram, is at best rugged and oscillatory. So, while an advantage of constructing a histogram is its simplicity, the two main disadvantages are (i) the discontinuity, or the ``ruggedness'', of the estimated density function and (ii) the difficulty of selecting an appropriate bin width.

A rather more modern approach to density estimation is the {\em kernel method}, which was introduced by Rosenblatt~\cite{Rosenblatt1956} in 1956 and has since been extensively investigated in numerous studies.  With this method, an estimator is computed using the sum of a kernel (function) expressed at each data point.  The expression for the estimator also involves a parameter, called the ``bandwidth'', adjusted to obtain a smoother appearance of the graph of the estimated density function.  Although the kernel method can achieve the continuity of the estimated function, the selection of the bandwidth to obtain a smooth-looking density function without compromising the accuracy of the estimate remains a challenging issue.  Given these pros and cons, the kernel method has become one of the most common approaches to nonparametric density estimation; see, for example, the popular statistical computing and graphics software~R~\cite{R, HotEve2014}. 

Yet another common approach is {\em maximum likelihood estimation} (MLE), which is typically used for {\em parametric density estimation}, where the parameters of a known form of density function (for example, the mean and variance of the normal distribution) are estimated, by maximizing (in some sense) the likelihood (or the probability) of the outcomes at the sampled data points $t_i$, $i=1,\ldots,n$; see, for example, \cite[Ch.~5]{Rossi2018}.   

For our focus of interest, i.e., for nonparametric density estimation, the MLE problem can be naively written as finding an estimate $f:[0,1]\to\dR$ that solves the maximization problem
\begin{equation}  \label{prob:MLE}
\max\ f(t_1)\,f(t_2)\,\cdots\, f(t_n)\,,
\end{equation}
subject to the {\em normality} constraint
\begin{equation}  \label{constr:normality}
\int_0^1 f(t)\,dt = 1\,.
\end{equation}
Taking the logarithm, $f$ equivalently solves
\begin{equation}  \label{prob:log_MLE}
\max\ \sum_{i = 1}^n \ln f(t_i)\,,
\end{equation}
subject to~\eqref{constr:normality}.  The so-called {\em $\log$-likelihood} expression in \eqref{prob:log_MLE} is commonly used for convenience in subsequent calculations in the literature.  However, without a specified or required form of the function $f$, there is not even a piecewise-continuous solution to \eqref{prob:log_MLE}.  As a remedy to this intractability, in 1971, Good and Gaskins~\cite{GooGas1971} considered adding a ``{\em nonparametric roughness penalty}'', or {\em regularization} terms involving the squared $L^2$-norms of derivatives $\zeta'$ and $\zeta''$, or in an equivalent notation $\dot{\zeta}$ and $\ddot{\zeta}$, of the function $\zeta := \sqrt{f}$, to the functional in~\eqref{prob:log_MLE}.  Then the regularized problem must be solved for the new function $\zeta$ instead of $f$, whose use avoids the constraint $f(t) \ge 0$.  In other words, they posed the following problem.
\begin{equation}  \label{prob:Good&Gaskins}
\min_{\gamma(\cdot)}\ -\sum_{i=1}^n \ln\zeta(t_i) 
+ \alpha_1 \int_0^1 \dot{\zeta}^2(t)\,dt 
+ \alpha_2 \int_0^1 \ddot{\zeta}^2(t)\,dt\,,
\end{equation}
subject to~\eqref{constr:normality}, where the constants $\alpha_1\ge 0$ and $\alpha_2\ge 0$, such that $\alpha_1 + \alpha_2 > 0$, are the so-called {\em penalty parameters}.  Good and Gaskins~\cite{GooGas1971} also presented a numerical scheme using the Rayleigh--Ritz method to solve~\eqref{prob:Good&Gaskins} subject to~\eqref{constr:normality} and illustrated their approach through examples.

More recently, Griebel and Hegland~\cite{GriHeg2010} have considered the problem of estimating a multivariate probability density function.  Here we cite the univariate version of the problem posed in~\cite{GriHeg2010} as follows.  Consider the problem of estimating the probability density function $f:[0,1]\to\dR$ by the function $f_v:[0,1]\to\dR$ such that
\begin{equation}  \label{fv}
f_v(t) := \frac{e^{v(t)}}{\ds\int_0^1 e^{v(\tau)}\,d\tau}\,,
\end{equation}
where $v:[0,1]\to\dR$ solves the problem
\[
\mbox{(P)\ \ }\min_{v(\cdot)}\
-\frac{1}{n} \sum_{i=1}^n v(t_i) 
+ \log \int_0^1 e^{v(t)}\,dt 
+ \frac{\alpha}{n}\,\int_0^1\dot{v}^2(t)\,dt 
+ \frac{\alpha\,\beta^2}{n} \int_0^1 (v(t) - w(t))^2\,dt\,.
\]
The function $w:[0,1]\to\dR$ above is continuous and refers to a known distribution.  For example, for the normal distribution, $w(t) = -(t-\mu)^2/(2\,\sigma^2)$, where $\mu$ is the mean and $\sigma^2$ is the variance.  In fact, \cite{GriHeg2010} considers the above problem with $w(t)=0$ for all $t\in[0,1]$, but we have incorporated $w(t)$ for a slightly more general setting.  

Minimization of the first two terms in Problem~(P) corresponds to nothing but maximum $\log$-likelihood.  We note that the constant $\alpha$ is the {\em smoothness parameter}, which is reminiscent of $\alpha_1$ in~\eqref{prob:Good&Gaskins}.  When $w=0$ the last two terms constitute a suitable (weighted) norm for a Sobolev space of functions and serve as regularization terms.  Minimization of the last term has to do with the ``structure'' (or the ``flatness'' in the case when $w = 0$) of the distribution, and so $\beta$ is referred to as the {\em structure parameter}.  With an appropriate choice of $\beta$, this last term ensures that the estimated density will not be too far from the density of some known distribution.  Problem~(P) is referred to as the {\em penalized maximum $\log$-likelihood} problem.

Furthermore, Griebel and Hegland~\cite{GriHeg2010} proposed a Newton--Galerkin method to solve Problem~(P) numerically and illustrated the method using numerical experiments utilizing various synthetic and real data sets.

\subsection{Contribution}

Problem~(P) is a calculus of variations problem and can be transformed into an optimal control problem, although it has a term that involves ``intermediate costs'' rather than an initial or a terminal cost, making it nonstandard.  In the present paper, we study an optimal control formulation of Problem~(P) using the maximum principle for problems with intermediate costs~\cite{AugMau2000,ClaVin1989,DmiKag2011}, and propose numerical methods to obtain approximate solutions to Problem~(P).  Our contribution can be described in more detail as follows.
\begin{itemize}
\item After reformulating Problem~(P) as a multiprocess or multistage optimal control problem, we present our main result in Theorem~\ref{theo:tpbvp}: the pertaining maximum principle, or necessary conditions of optimality, reduces to a two-point boundary-value problem described in the variable $v$ (defined in~\eqref{fv}) and an auxiliary variable, with jumps in the value of $v$ at data points $t_i$.
\item Theorem~\ref{theo:reduction} presents an auxiliary result that states an equivalent ODE for $v$ in reduced order.
\item Since $v$ has jumps at data points, we describe a novel discretization scheme, or partitioning, taking this into account, which makes use of either the Euler method or the trapezoidal rule.  
\item To solve the large-scale equation system resulting from discretization, we use the AMPL--Knitro suite~\cite{AMPL,Knitro}, where AMPL is an optimization modelling language that employs Knitro as the solver.  We illustrate our approach with synthetic data from normal distribution, as well as practical (or real) data from the Old Faithful geyser and a group of galaxies.
\item We employ the popular statistical and graphics software R~\cite{R} to make comparisons with our approach and conclude that the density estimates obtained by our approach are at least as good as those obtained by R.
\end{itemize}

It is worth mentioning that Shvartsman~\cite{Shvartsman2010} studied earlier an optimal control formulation of the problem in~\eqref{prob:log_MLE} subject to~\eqref{constr:normality} and the ``modifications'' that $f$ is Lipschitzian with a known Lipschitz constant and that $f$ is nonnegative.  The Lipschitzianity of $f$ in the modified problem in~\cite{Shvartsman2010}, posed as the constraint $|f(t)| \le \ell$, for all $t\in[0,1]$, with the specified Lipschitz constant $\ell$, has the regularization effect that results in the existence of a solution.  As can be seen, this problem is markedly different from Problem~(P) in the way regularization is achieved.

The main concern in~\cite{Shvartsman2010} is to present convergence results (asymptotic with the data size $n$).  Although some structure of the solutions are elaborated, such as the bang--bang nature of the optimal control, or the seesaw appearance of the graph of the estimated density function, no computational method is proposed in~\cite{Shvartsman2010} to implement these results.  We stress that especially when $n$ is not large, the seesaw appearance of the density function is not desirable either for what we want to achieve in this paper.  Moreover, the approach in~\cite{Shvartsman2010} requires good knowledge of the Lipschitz constant $\ell$ which may not be so straightforward to estimate.  Therefore, we rather focus on Problem~(P) which produces density estimates with ``smoother'' graphs.

The paper is organized as follows.  In Section~\ref{sec:formulation}, we formulate Problem~(P) as a multistage optimal control problem.  We apply a maximum principle to this optimal control problem in Section~\ref{sec:maxprinciple} and establish the normality of the problem in Lemma~\ref{lem:normal}. The main results of the paper, namely Theorems~\ref{theo:tpbvp} and \ref{theo:reduction} are presented in Section~\ref{sec:results}.  In Section~\ref{sec:numerical}, the discretization scheme and numerical experiments for synthetic and real data are provided.  Finally, Section~\ref{conclusion} presents concluding remarks and comments for future work.

\section{Optimal Control Formulation}
\label{sec:formulation}

Problem~(P) is unconstrained, which is preferable as a general variational problem.  On the other hand, optimal control theory can handle certain constraints with ease.  Therefore, we pose the following natural constraint, which we express in our optimal control framework as a terminal state constraint.
\begin{equation} \label{constr1}
\int_0^1 e^{v(\tau)}\,d\tau = 1\,.
\end{equation}
Now, obviously,
\begin{equation}  \label{fv2}
f_v(t) = e^{v(t)}\,,
\end{equation}
which is simpler.

Let $x_1(t) := \ds\int_0^t e^{x_2(t)}\,dt$, where $x_2(t) := v(t)$.  We refer to the functions $x_1:[0,1]\to\dR$ and $x_2:[0,1]\to\dR$ as the {\em state variables}. Let $u := \dot{v}$. We refer to the function $u:[0,1]\to\dR$ as the {\em control variable}.  Now, Problem~(P), along with the constraint \eqref{constr1}, can be written equivalently as the optimal control problem
\[
\mbox{(OCP) }\left\{\begin{array}{rl} 
\ds\min & \ \ \ds - \frac{1}{\alpha} \sum_{i=1}^n x_2(t_i) +
\frac{1}{2}\int_0^1 \left[\beta^2\,(x_2(t) - w(t))^2 +
  u^2(t)\right]\,dt \\[4mm] 
\mbox{subject to} & \ \ \dot x_1(t) = e^{x_2(t)}\,,\quad\ x_1(0) =
0\,,\quad x_1(1) = 1\,, \\[2mm]
& \ \ \dot x_2(t) = u(t)\,,
\end{array} \right.
\]
where $\dot{x}_i := dx_i/dt$, $i=1,2$.  The first term in the objective functional in Problem~(OCP) involves values of $x_2$ at discrete points in the interior of the time horizon $[0,1]$, which makes the optimal control problem nonstandard.  We note that $x_1$ is the {\em cumulative density function} of the distribution.

Suppose that $t_i\neq t_j$ for $i\neq j$, $i,j=1,\ldots,n$.  Without loss of generality, order the sample points $t_i$ such that
\[
0 = t_0 < t_1 < t_2 < \cdots < t_n < t_{n+1} = 1\,.
\]
Now, Problem~(OCP) can be reformulated as a so-called multiprocess, or multistage, optimal control problem for which the necessary conditions of optimality can be derived following the theory and methodology provided in~\cite{AugMau2000,ClaVin1989,DmiKag2011}.  In the reformulation, the process over an interval $[t_{i-1},t_i]$, is referred to as {\em stage $i$}, $i = 1,\ldots,n+1$.  We point out the papers \cite{KayNoa2013,Kaya2019}, where somewhat similar reformulations were employed for interpolation problems, in writing the necessary conditions of optimality.

Clarke and Vinter give a maximum principle for multiprocess (or multistage) optimal control problems in \cite{ClaVin1989} involving very general dynamical systems, including systems which are not differentiable, for which the transversality conditions are presented by means of generalized derivatives and normal cones.  Augustin and Maurer transform in \cite{AugMau2000} the multistage control problem for a special class of systems (including the class of system we have in this paper) into a single-stage one by means of a standard rescaling of the stage durations (defined below).  This allows the transversality conditions to be described in a rather more convenient way.

Define a new time variable $s$ in each stage in terms of $t$ as follows:
\[
s := (t-t_{i-1}) / (t_i - t_{i-1})\,,\quad\mbox{for}\ \
t\in[t_{i-1},t_{i}]\,, 
\]
and all $i = 1,\ldots,n+1$.  With this definition, each stage
$[t_{i-1},t_i]$ is re-scaled as $[0,1]$ in the new time variable~$s$.
Let
\[
x_j^{[i]}(s) := x_j(t)\quad\mbox{and}\quad u^{[i]}(s) :=
u(t)\quad\mbox{for}\ \ s\in[0,1],\ t\in[t_{i-1},t_{i}]\,,
\]
$j = 1,2$, and all $i = 1,\ldots,n+1$.  Here, $x_j^{[i]}$ and $u^{[i]}$ denote the values of the state and control variables $x_j$ and $u$, respectively, in stage $i$.  In addition to the ``interior'' point objective function terms in (OCP), with the usage of stages, one needs to pose constraints to ensure continuity of the state variables at the junctions of two successive stages; namely one should require
\[
x_j^{[i]}(1) = x_j^{[i+1]}(0)\,,\quad\mbox{for} \ \ j = 1,2\,,
\]
and all $i = 1,\ldots,n+1$.

\section{A Maximum Principle for Density Estimation}
\label{sec:maxprinciple}

We will make use of both \cite[Section~4]{AugMau2000} and \cite[Theorem~3.1]{ClaVin1989} to write the necessary conditions of optimality for Problem~(OCP).  We re-iterate that the maximum principle for optimal multiprocesses from these references has also been implemented in \cite{KayNoa2013, Kaya2019} for interpolation problems.

Define the Hamiltonian function in the $i$th stage as
\[
H^{[i]}(x_1^{[i]},x_2^{[i]},u^{[i]},\lambda_0,\lambda_1^{[i]},\lambda_2^{[i]},s)
:= \frac{1}{2}\,\lambda_0\,\left(\beta^2\,(x_2^{[i]} - w(s))^2 +
  (u^{[i]})^2\right) + \lambda_1^{[i]}\, e^{x_2^{[i]}} +
\lambda_2^{[i]}\, u^{[i]}\,, 
\]
where $\lambda_0$ is a real scalar (multiplier) parameter and
$\lambda_1^{[i]},\lambda_2^{[i]}:[0,1]\rightarrow\dR$ are the
adjoint variables (multipliers) in the $i$th stage.  Let
\[
H^{[i]}[s] :=
H^{[i]}(x_1^{[i]}(s),x_2^{[i]}(s),u^{[i]}(s),\lambda_0,\lambda_1^{[i]}(s),\lambda_2^{[i]}(s),s)\,.
\]

Suppose that $x_1,x_2\in W^{1,\infty}(0,1;\dR)$, $u\in
L^\infty(0,1;\dR)$, are optimal for Problem~(OCP).  Then there exist
a number $\lambda_0\ge0$, functions $\lambda_1^{[i]},\lambda_2^{[i]}\in
W^{1,\infty}(0,1;\dR)$, such that $\lambda^{[i]}(t)=(\lambda_0,
\lambda_1^{[i]}(s), \lambda_2^{[i]}(s))\neq\bf0$, $i= 1,\ldots,n+1$,
for every $s\in[0,1]$, and the following conditions hold, in addition
to the constraints given in Problem~(OCP).
\begin{subequations}
\begin{eqnarray}
&& \dot\lambda_1^{[i]}(s) = -H^{[i]}_{x_1^{[i]}}[s] = 0\,,\quad
i = 1,\ldots,n+1\,, \label{adjoint1} \\[1mm]
&& \dot\lambda_2^{[i]}(s) = -H^{[i]}_{x_2^{[i]}}[s] =
-\lambda_0\,\beta^2\,(x_2^{[i]}(s) - w(s)) - \lambda_1^{[i]}\,
e^{x_2^{[i]}(s)}\,,\quad i = 1,\ldots,n+1\,, 
  \label{adjoint2} \\[1mm]
&& \lambda_1^{[i]}(1) = \lambda_1^{[i+1]}(0)\,,\quad
  \ i = 1,\ldots,n\,, \label{contin} \\[1mm]
&& \lambda_2^{[i]}(0)=0\,,\ \lambda_2^{[i]}(1)=0\,,\quad \ i =
1,n\,, \label{bc} \\[1mm]
&& \lambda_2^{[i+1]}(0) = \lambda_2^{[i]}(1) + 1/\alpha\,,\quad
  i=1,\ldots,n\,,
  \label{jump} \\[1mm]
&& 0 =
H^{[i]}_{u^{[i]}}(x_1^{[i]}(s),x_2^{[i]}(s),u^{[i]},\lambda_0,\lambda_1^{[i]}(s),\lambda_2^{[i]}(s))\,,\quad
i = 1,\ldots,n+1\,.  \label{controla}
\end{eqnarray}
\end{subequations}
The notation $H^{[i]}_{x_j^{[i]}}$ denotes the partial derivative of
$H^{[i]}$ with respect to $x_j^{[i]}$, $j=1,2$, and
$H^{[i]}_{u^{[i]}}$ the partial derivative of $H^{[i]}$ with respect
to $u^{[i]}$.  Conditions~\eqref{adjoint1} and \eqref{contin} imply
that the value of the adjoint variable $\lambda_1^{[i]}$ in each stage
is the same constant. Condition \eqref{jump}, on the other hand,
asserts a jump of a fixed amount of $1/\alpha$ in the value of
$\lambda_2$, at the junctions of the stages.

For a neater appearance, we will re-write Conditions~\eqref{adjoint1}--\eqref{controla} and elaborate them further by
means of the {\em general} state, control, and adjoint variables. For
this purpose, define the general adjoint variables $\lambda_1(t)$ and $\lambda_2(t)$ formed by concatenating the stage adjoint variables, as
follows.
\[
\lambda_j(t) := \lambda_j^{[i]}(s)\,,\quad  t = t_{i-1} + s\,\tau_i,\quad
s\in[0,1]\,,\quad \tau_i := t_i - t_{i-1}\,,\quad i = 1,\ldots,n+1\,,\ \ j = 1,2\,.
\]
The general state and control variables are defined in a similar
way.  Conditions \eqref{adjoint1}--\eqref{controla}, along with
the state equations and constraints, can now be neatly re-written as
follows.
\begin{subequations}
\begin{eqnarray}
&& \dot x_1(t) = e^{x_2(t)}\,,\quad\ x_1(0) = 0\,,\quad x_1(1) =
1\,,  \label{state1} 
\\[1mm]
&& \dot x_2(t) = u(t)\,,  \label{state2} \\[1mm]
&& \lambda_1(t) = \gamma\,,  \label{adjoint1b} \\[1mm] 
&& \dot\lambda_2(t) = -\lambda_0\,\beta^2\,(x_2(t) - w(t)) -
\gamma\,e^{x_2(t)}\,, \mbox{ for a.e. } t\in[0,1]\,, \nonumber
\\[1mm]
&& \hspace*{20mm} \lambda_2(0) = 0\,,\
  \lambda_2(t_i^+) = \lambda_2(t_i^-) + \frac{1}{\alpha}\,,\
  \lambda_2(1) = 0,  \label{adjoint2b} \\[1mm]
&& \lambda_0\,u(t) = -\lambda_2(t)\,,  \label{controlb}
\end{eqnarray}
\end{subequations}
where $\gamma$ is an unknown constant.

The problems which result in $\lambda_0 = 0$ are called {\em abnormal} in the optimal control theory literature, for which the necessary conditions in \eqref{state1}--\eqref{controlb} are independent of the objective functional and therefore not fully informative.  The problems that result in $\lambda_0 > 0$ are referred to as {\em normal}.  Lemma~\ref{lem:normal} below asserts that Problem~(OCP) is normal.
\begin{lemma}[Normality]  \label{lem:normal} 
One has that $\lambda_0 > 0$, i.e., that Problem~(OCP) is normal.  In particular, one can take $\lambda_0 = 1$, and so the optimal control can be written as $u(t) = -\lambda_2(t)$.
\end{lemma}
\begin{proof}
Suppose that $\lambda_0 = 0$.  Then \eqref{controlb} implies that $\lambda_2(t) = 0$ for a.e. $t\in[0,1]$, and thus, by the differential equation in \eqref{adjoint2b}, $\gamma=0$.  Therefore,
one gets $\lambda^{[i]}(t)=(\lambda_0, \lambda_1^{[i]}(s), \lambda_2^{[i]}(s)) = \bf0$, $i= 1,\ldots,n+1$, which is not allowed by the maximum principle.  As a result, $\lambda_0 > 0$.  Any positive scalar multiple of $\lambda_0$ is also a solution.  Therefore, without loss of generality, one can set $\lambda_0 = 1$, and so write the optimal control from \eqref{controlb} as $u(t) = -\lambda_2(t)$.
\end{proof}

\begin{remark}[Jumps in Optimal Control] \rm
With $\lambda_0 = 1$ by Lemma~\ref{lem:normal}, we note that the jump condition in~\eqref{adjoint2b} and Equation~\eqref{controlb} implies that the optimal control $u$ has jumps at $t_i$, namely that $u(t_i^+)=u(t_i^-)-1/\alpha$, $i=1,\ldots,n$.
\proofbox
\end{remark}

\section{Main Results}
\label{sec:results}

The ultimate result of the paper is furnished by the theorem below, which presents a two-point boundary-value problem with interior jump conditions that is required to be solved by the function $v$ in~\eqref{fv}.
\begin{theorem}[Necessary Condition of Optimality]  \label{theo:tpbvp} 
If the function $v$ in \eqref{fv} solves Problem~{\em (P)} then it solves the two-point boundary-value problem (with interior jump conditions)
\begin{subequations}
\begin{eqnarray}
&& \dot{z}(t) = e^{v(t)}\,,\hspace*{36mm} z(0) = 0\,,\ \ z(1) =
1\,,  \label{tpbvp1}  \\[1mm]
&& \ddot{v}(t) = \beta^2\,(v(t) - w(t)) + \gamma\,e^{v(t)}\,,\ \ \dot{v}(0) = 0\,,\ \ \dot{v}(1) = 0\,, \nonumber \\
&& \hspace*{57mm} \dot{v}(t_j^+) = \dot{v}(t_j^-) - \frac{1}{\alpha}\,,\ j=1,\ldots,n\,,  \label{tpbvp2}
\end{eqnarray}
\end{subequations}
for all $t\in(t_i,t_{i+1})$, $ i = 0,\ldots,n$, where $t_0 = 0$, $t_{n+1} = 1$, and $\gamma$ is a real constant.
\end{theorem}
\begin{proof}
Equation~\eqref{state2} and Lemma~\ref{lem:normal} imply that $\dot{v}(t) = \dot{x}_2(t) = u(t) = -\lambda_2(t)$.  Then $\ddot{v}(t) = -\dot{\lambda}_2(t)$.  Substituting $x_1 = z$, $x_2 = v$, $\lambda_2 = -\dot{v}$ and $\dot{\lambda}_2 = -\ddot{v}$ into~\eqref{state1} and \eqref{adjoint2b}, and rearranging, one gets the two-point boundary-value problem stated in \eqref{tpbvp1}--\eqref{tpbvp2}.
\end{proof}

\begin{remark}[Smoothness Parameter {\boldmath $\alpha$}]  \label{rem:jump}  \rm 
Since only a finite jump occurs in the values of $\dot{v}$ at $t_i$, the variable $v$ is continuous in $t$.  Therefore, $f_v$ in~\eqref{fv2} is not differentiable but continuous at $t_i$.  Otherwise, $f_v$ is continuously differentiable at all $t\neq t_i$.  It should be noted that, as the smoothness parameter $\alpha$ tends to infinity, the jump $1/\alpha$ in the values of $\dot{v}$ at $t_i$ tends to zero, in other words, as $\alpha\to\infty$, $\dot{v}(t_i^+)\to\dot{v}(t_i^-)$.  Likewise, it is no wonder why there exists no solution if $\alpha\to 0$, as the jumps at $t_i$ then tend to infinity.

It is worth pointing out that, with a finite $\alpha > 0$, {\em smoothness} is never achieved per se, mathematically speaking; however, one rather gets ``closer'' to a ``smooth solution'' with larger values of $\alpha$.
\proofbox
\end{remark}

The following corollary to Theorem~\ref{theo:tpbvp} provides an expression for the optimal value of $\gamma$.
\begin{corollary}[Parameter {\boldmath $\gamma$}]
\label{cor:gamma}
One has the identity that
\begin{equation}  \label{id:gamma}
\gamma = \frac{n}{\alpha} - \beta^2\int_0^1 (v(\tau) - w(\tau))\,d\tau\,.
\end{equation}
\end{corollary}
\begin{proof}
Integrate both sides of the ODE in~\eqref{tpbvp2} to get
\begin{equation}  \label{int:gamma}
\int_0^1 \ddot{v}(\tau)\,d\tau = \beta^2\int_0^1 (v(\tau) -w(\tau))\,d\tau + \gamma\int_0^1 e^{v(\tau)}\,d\tau\,.
\end{equation}
The left-hand side of~\eqref{int:gamma} can be expanded, and evaluated using the boundary and interior conditions in~\eqref{tpbvp2}, as follows.
\begin{eqnarray}
\int_0^1 \ddot{v}(\tau)\,d\tau &=& \int_0^{t_1} \ddot{v}(\tau)\,d\tau + \int_{t_1}^{t_2} \ddot{v}(\tau)\,d\tau + \ldots + \int_{t_n}^{1} \ddot{v}(\tau)\,d\tau \nonumber \\
&=& (\dot{v}(t_1^-) - \dot{v}(0)) + (\dot{v}(t_2^-) - \dot{v}(t_1^+)) + (\dot{v}(t_3^-) - \dot{v}(t_2^+)) + \ldots + (\dot{v}(1) - \dot{v}(t_n^+)) \nonumber \\[1mm]
&=& - \dot{v}(0) + (\dot{v}(t_1^-) - \dot{v}(t_1^+)) + (\dot{v}(t_2^-) - \dot{v}(t_2^+)) + \ldots + (\dot{v}(t_n^-) - \dot{v}(t_n^+)) + \dot{v}(1) \nonumber \\
&=& 0 + \frac{1}{\alpha} + \frac{1}{\alpha} + \ldots + \frac{1}{\alpha} + 0 \nonumber \\
&=& \frac{n}{\alpha}\,. \label{LHS}
\end{eqnarray}
On the other hand, by~\eqref{tpbvp1}, the second integral on the right-hand side of~\eqref{int:gamma} can simply be evaluated as
\begin{equation} \label{RHS}
\int_0^1 e^{v(\tau)}\,d\tau = z(1) - z(0) = 1\,.
\end{equation}
Now, substituting \eqref{LHS}--\eqref{RHS} into \eqref{int:gamma} and rearranging the terms, one gets~\eqref{id:gamma}.
\end{proof}

\begin{remark}[Asymptotic value of {\boldmath $\gamma$}] \rm
\label{rem:gamma}
Suppose that the modulus of the integral in~\eqref{id:gamma} is bounded by some constant $M>0$ for all $n \ge N$, where $N$ is a positive integer.  Then it follows from~\eqref{id:gamma} that, for any given $\varepsilon >0$, there exists a large enough $n$ such that $|\gamma - n/\alpha| \le\varepsilon$.
In other words, practically speaking, if the difference between the solution function $v(t)$ and the given function $w(t)$ remains bounded, then, for large enough $n$, $\gamma$ will be approximately equal to $n/\alpha$.
\end{remark}

If we take $w(t)=0$, for all $t\in[0,1]$, then the order of the ODE in \eqref{tpbvp2} can be reduced, as we state below.  Although this result does not provide an additional practical or theoretical advantage, we still state it here for the sake of completeness.

\begin{theorem}[Order Reduction]
\label{theo:reduction}
Suppose that $w(t)=0$ for all $t\in[0,1]$.  Then~\eqref{tpbvp2} can be replaced by
\begin{equation}  \label{reduced}
\dot{v}^2(t) = \beta^2\,v^2(t) + 2\,\gamma\,e^{v(t)} + C_i\,,
\quad\mbox{for all } t\in[t_i,t_{i+1}),\ 
i=0,1,\ldots,n\,, 
\end{equation}
with $\dot{v}(1) = 0$\,, where $\gamma$ is some real number, and
\[
C_i = C_{i-1} - \frac{2}{\alpha}\,\dot{v}(t_i^-) + \frac{1}{\alpha^2}\,,
\]
for $i=1,\ldots,n$, with
\[
C_0 = -\beta^2\,v^2(0) - 2\,\gamma\,e^{v(0)}\,.
\]
\end{theorem}
\begin{proof}
  In order to reduce the order of the differential equation in
  \eqref{tpbvp2}, one can use the transformation $\phi(v):=\dot{v}$
  (see, e.g., \cite[Section 2.9.1]{ZaiPol2003}), because the
  right-hand side of \eqref{tpbvp2} does not depend on $t$ explicitly.
  Note that $d\phi(v)/dt = \ddot{v}$ and so, using~\eqref{tpbvp2},
\[
\frac{d\phi}{dv}\,\phi =
\frac{d}{dv}\,\left(\frac{1}{2}\,\phi^2\right) = \beta^2\,(v - w)+
\gamma\,e^v\,,
\]
which, after integrating, yields \eqref{reduced}, with real constants $C_i$. For $t\in[0,t_1)$, $C_0$ is simply obtained by substituting
$t=0$ and $\dot{v}(0) = 0$ into \eqref{reduced} and then re-arranging
the resulting equation.  For $t\in[t_i,t_{i+1})$, $i=1,\ldots,n$,
$C_i$ in \eqref{reduced} are obtained as follows.  Note that, using~\eqref{reduced} and the interior conditions in \eqref{tpbvp2}, one
gets
\begin{eqnarray*}
\dot{v}^2(t_i^+) = \left(\dot{v}(t_i^-) - \frac{1}{\alpha}\right)^2
&=& \dot{v}^2(t_i^-) - \frac{2}{\alpha}\,\dot{v}(t_i^-) +
\frac{1}{\alpha^2} \\
&=& \beta^2\,v^2(t_i) + 2\,\gamma\,e^{v(t_i)} + C_i\,.
\end{eqnarray*}
Now, the substitution of
\[
\dot{v}^2(t_i^-) = \beta^2\,v^2(t_i) + 2\,\gamma\,e^{v(t_i)} + C_{i-1} 
\]
into the above equation, and manipulations, result in the
expression required for $C_i$.
\end{proof}

\section{Numerical Implementation and Experiments}
\label{sec:numerical}

In estimating the density function in~\eqref{fv}, or~\eqref{fv2}, the two-point boundary value problem with $n$ specified interior points
in~\eqref{tpbvp1}--\eqref{tpbvp2} must be solved numerically.  In what follows, we propose a novel discretization scheme (partitioning) and standard (Euler and trapezoidal) methods for this purpose.

\subsection{Discretization}

Let $y_1:=z$, $y_2:=v$ and $y_3:=\dot{v}$.  Then \eqref{tpbvp1}--\eqref{tpbvp2}
can be re-written as the system of first-order ODEs with boundary and interior conditions and an unknown parameter, as
\begin{subequations}
\begin{eqnarray} 
&& \dot{y}_1(t) = e^{y_2(t)}\,,\hspace*{38mm} y_1(0) = 0\,,\ \ y_1(1) = 1\,,  \label{tpbvp1a}  \\[1mm]
&& \dot{y}_2(t) = y_3(t)\,, \label{tpbvp2a} \\[1mm]
&& \dot{y}_3(t) = \beta^2\,(y_2(t) - w(t)) + \gamma\,e^{y_2(t)}\,,\ \ y_3(0) = 0\,,\ \ y_3(1) = 0\,, \nonumber \\
&& \hspace*{62mm} y_3(t_j^+) = y_3(t_j^-) - \frac{1}{\alpha}\,,\ j=1,\ldots,n\,,  \label{tpbvp3a}
\end{eqnarray}
\end{subequations}
for all $t\in(t_i,t_{i+1})$, $ i = 0,\ldots,n$, where $t_0 = 0$, $t_{n+1} = 1$, and $\gamma$ is the unknown parameter.

Let $h$ denote the {\em nominal step size} of the discretization.  Then the number of steps
$m_i$ in stage $i$ is given by
\[
m_i := \left\lceil \frac{t_i - t_{i-1}}{h}\right\rceil\,,
\]
$i = 1,\ldots,n+1$, where $\lceil\,\cdot\,\rceil$ denotes the smallest greater integer.  We define the {\em step sizes} in stage~$i$, $i=1,\ldots,n+1$, as
\[
h_j^i := \left\{\begin{array}{ll}
h\,,   &\ \mbox{if } m_i > 1\,,\mbox{ for } j = 1,\ldots,m_i-1\,, \\[1mm] 
t_i - h_{m_i-1}^i\,,   &\ \mbox{if } m_i > 1\,,\mbox{ for } j = m_i\,, \\[1mm] 
t_i - t_{i-1}\,, &\ \mbox{if } m_i = 1\,,\mbox{for } j = m_i.
\end{array} \right.
\]
The above step sizes inform one as to how the discretization (time grid) points should next be defined.  Let $t_{i,0} := t_i$ and, if $m_i > 1$, $t_{i,j} := t_{i,j-1} + h$, $i
= 0,\ldots,n+1$, $j = 1,\ldots,m_i$.  Define the partition
\begin{equation}  \label{partition}
\pi := \{t_{0,0},t_{0,1},\ldots,t_{0,m_1-1};t_{1,0},t_{1,1},\ldots,t_{1,m_2-1};\ldots;t_{n,1},t_{n,2},\ldots,t_{n,m_{n+1}-1};t_{n+1,0}\}\ .
\end{equation}
Next, define the two index sets,
\begin{equation}  \label{indexsets}
K := \{0,1,\ldots,L\}\qquad\mbox{and}\qquad
T := \left\{m_1, m_1 + m_2,\ldots,\sum_{i=1}^n m_i\right\}\,,
\end{equation}
where $L=\card(\pi)- 1$, with $\card(\pi)$ denoting the cardinality
(i.e., the number of elements) of the partition set $\pi$.  Note that
the elements of the index set $T$ correspond to (the subscripts of) the sample points
$t_1,t_2,\ldots,t_n$: Re-write the set $\pi$, with its elements
re-named, as $\pi = \{s_0,\ldots,s_L\}$.  Then $t_1 = s_{m_1}$, $t_2 =
s_{m_1+m_2}$, and so on. We also conveniently define a sequence of step sizes by
\begin{equation}  \label{stepsize}
h_k := s_{k+1} - s_k\,,\quad k=0,\ldots,L\,.
\end{equation}

For brevity, let the right-hand side of the ODE in~\eqref{tpbvp3a} be defined as
\begin{equation} \label{def:f}
f(y_2(t),t) := \beta^2\,(y_2(t) - w(t)) + \gamma\,e^{y_2(t)}\,.
\end{equation}
Then the {\em Euler discretization} of the equations in~\eqref{tpbvp1a}--\eqref{tpbvp3a}, incorporating the definition in~\eqref{def:f}, is described by
\begin{subequations}
\begin{eqnarray} 
&& \left[\begin{array}{l} y_{1,k+1} \\[1mm]  y_{2,k+1} \\[1mm] y_{3,k+1} \end{array}\right] = 
\left[\begin{array}{l} y_{1,k} \\[1mm]  y_{2,k} \\[1mm] y_{3,k} \end{array}\right]  + 
\left\{\begin{array}{ll}
h_k \left[\begin{array}{l} e^{y_{2,k}} \\[1mm] y_{3,k}  \\[1mm] f(y_{2,k},s_k)\end{array}\right],   &\mbox{if } k\in K\backslash T\,, \\[8mm] 
\left[\begin{array}{l} h_k\,e^{y_{2,k}} \\[1mm]  h_k\,(y_{3,k} - 1/\alpha) \\[1mm]  -1/\alpha + h_k\,f(y_{2,k},s_k)\end{array}\right], &\mbox{if } k\in T\,;
\end{array} \right.  \label{Euler1}\\[2mm]
&& y_{1,0} = 0\,,\ \ y_{1,L} = 1\,,\ \  y_{3,0} = 0\,,\ \ y_{3,L} = 0\,,  \label{Euler2}
\end{eqnarray}
\end{subequations}
for $k=0,1,\ldots,L-1$.  In \eqref{Euler1}--\eqref{Euler2}, $y_{i,k}$
are the Euler scheme approximations of $y_i(s_k)$, $i=1,2$.

A solution of the Euler approximation of Equations~\eqref{tpbvp1a}--\eqref{tpbvp3a} has an accuracy of order one.  For higher-order accuracies, more general (implicit as well as explicit) Runge--Kutta methods can also be employed~\cite{HaiLubWan2006}; however, in that case, one needs to generate the partition with more care because of the jump in the values of $\dot{v}$ at the data points $t_k$, $k\in T$.  We note that the Euler scheme is, in fact, the simplest possible (explicit) Runge--Kutta method.

In what follows, we provide an order-two approximation of \eqref{tpbvp1a}--\eqref{tpbvp3a}, along with the boundary conditions, using the {\em trapezoidal rule}, which is an implicit Runge--Kutta method.  The trapezoidal rule requires evaluation at just two points; therefore, the discontinuities at $t_k$, $k\in T$, do not pose any difficulty, just as in the case of Euler's method, and these discontinuities can be managed efficiently using the partitioning defined in \eqref{partition}--\eqref{stepsize}.
\begin{subequations}
\begin{eqnarray} 
&& \left[\begin{array}{l} y_{1,k+1} \\[1mm]  y_{2,k+1} \\[1mm] y_{3,k+1} \end{array}\right] = 
\left[\begin{array}{l} y_{1,k} \\[1mm]  y_{2,k} \\[1mm] y_{3,k} \end{array}\right]  + 
\left\{\begin{array}{ll}
\ds\frac{h_k}{2} \left[\begin{array}{l} e^{y_{2,k}} + e^{y_{2,k+1}} \\[1mm] y_{3,k} + y_{3,k+1}  \\[1mm] f(y_{2,k},s_k) + f(y_{2,k+1},s_{k+1})\end{array}\right],   &\mbox{if } k\in K\backslash T\,, \\[8mm] 
\left[\begin{array}{l} h_k\,(e^{y_{2,k}} + e^{y_{2,k+1}})/2 \\[1mm]  h_k\,(y_{3,k} - 1/\alpha + y_{3,k+1})/2 \\[1mm]  -1/\alpha + h_k\,(f(y_{2,k},s_k) + f(y_{2,k+1},s_{k+1}))/2\end{array}\right], &\mbox{if } k\in T\,;
\end{array} \right.  \label{trap1}\\[2mm]
&& y_{1,0} = 0\,,\ \ y_{1,L} = 1\,,\ \  y_{3,0} = 0\,,\ \ y_{3,L} = 0\,,  \label{trap2}
\end{eqnarray}
\end{subequations}
for $k=0,1,\ldots,L-1$.  In \eqref{trap1}--\eqref{trap2}, $y_{i,k}$ are
the trapezoidal rule approximations of $y_i(s_k)$, $i=1,2$.  

In view of the intricacies of the management of discontinuities hinted above, we leave higher-order approximations via more general Runge--Kutta methods outside the scope of the current paper.  Moreover, not employing higher-order Runge--Kutta schemes here is also justified because the estimate of the density function itself (as a solution of \eqref{tpbvp1a}--\eqref{tpbvp3a}) is not even differentiable, for finite $\alpha$.

\subsection{Numerical experiments}

The Euler discretization given in \eqref{Euler1}--\eqref{Euler2}, or the trapezoidal discretization given in \eqref{trap1}--\eqref{trap2}, constitute a nonlinear system of $3L+4$ equations in the $3L+4$ unknowns, $y_{i,j}$, $i = 1,2,3$, $j=0,1,\ldots,L$, and $\gamma$.  Here, $L$ is typically large since it has to be at least as large as the number of data points $n$ and we aim to get a reasonably accurate approximation of the density function estimate.  For solving either of the large-scale system of equations as a {\em feasibility problem}, various well-established general nonlinear programming software are available, such as Algencan~\cite{Andreani2007,BirMar2014}, which implements augmented Lagrangian techniques; Ipopt~\cite{WacBie2006}, which implements an interior point method; SNOPT~\cite{GilMurSau2005}, which implements a sequential quadratic programming algorithm; Knitro~\cite{Knitro}, which implements various interior point and active set algorithms to choose from.

In order to obtain an approximate, or discrete, solution to~\eqref{tpbvp1}--\eqref{tpbvp2}, we employ the solver Knitro (version 13.0.1 is used here) and use AMPL \cite{AMPL} as a modelling language that employs Knitro as the solver.  We set the Knitro parameters {\tt alg=0} (meaning that it is left to Knitro to choose an appropriate algorithm) and {\tt feastol=1e-10} (meaning that we set the feasibility tolerance at $10^{-10}$).

The {\em CPU times} are reported through the AMPL command {\tt \_ampl\_elapsed\_time}, the AMPL--Knitro suite running on a 14-inch 2021-model MacBook Pro, with the operating system macOS Sequoia (version 15.2), the Apple M1 Max processor with a 10 core CPU and the 64 GB LPDDR5 memory.

For comparison purposes, we also employ the statistical computing and graphics software~R~\cite{R}, version 4.4.2 (Pile of Leaves) released on 31 October 2024, which finds estimates of a density function for given data using the kernel method, for various {\em bandwidths}, abbreviated as ``bw'' in numerical experiments here.

In Examples~2 and 3, we display the histograms of the data provided by using {\sc Matlab}'s {\tt histogram} command~\cite{Matlab}, which automatically chooses an appropriate number of bins to cover the range of values in the data set.

\subsubsection{Example 1: A normal distribution}

We consider data sets of various sizes, namely $n = 10^2, 10^3, 10^4, 10^5$, data sampled from a normal distribution defined over a domain of $[0,1]$, with mean $\mu = 0.5$ and variance $\sigma^2 = 0.01$.  In Figure~\ref{fig:normal}, we display, for each data set, the graphs of the estimated density function from solutions to \eqref{trap1}--\eqref{trap2} (which are approximate solutions to \eqref{tpbvp1}--\eqref{tpbvp2}) for $\beta=1$ and various values of $\alpha$.  In Table~\ref{table:normal}, we list, for each data set, the nominal step size $h$ used for the domain partition, the resulting number $L$ of (the trapezoidal rule's) discretization points (or nodes), the set of values used for the smoothing parameter $\alpha$, and the resulting solution value of the constant $\gamma$ by solving \eqref{trap1}--\eqref{trap2}.  We also report the CPU times taken to obtain a solution.

\begin{table}
\centering
{\small
\begin{tabular}{cccccc} \hline \\[-2mm]
$n$ & $h$ & $L$ & $\alpha$ & $\gamma$ & \hspace*{-3mm}$\begin{array}{c} \mbox{CPU time} \\ \mbox{[sec]} \end{array}$ \\[4mm] \hline \\[-3mm]
$10^2$ & $1/(2\times10^3)$ & 2,052 & $\begin{array}{r} 0.01 \\ 0.1 \\ 1\ \ \\ 3\ \ \end{array}$ & $\begin{array}{r} 10002.4 \\ 1001.3 \\ 100.4 \\ 33.5 \end{array}$ & 0.11 \\[1mm] \hline \\[-3mm]
$10^3$ & $1/(2\times10^3)$ & 2,612 & $\begin{array}{r} 0.1 \\ 1\ \ \\ 5\ \ \end{array}$ & $\begin{array}{r} 10001.8 \\ 1001.1 \\ 200.6 \end{array}$ & 0.17 \\[1mm] \hline  \\[-3mm]
$10^4$ & $1/(2\times10^4)$ & 26,152\ \ & $\begin{array}{r} 1  \\ 10 \\ 30 \end{array}$ & $\begin{array}{r} 10001.8 \\ 1001.1 \\ 334.1 \end{array}$ & 92\ \ \ \ \ \ \ \\[1mm] \hline  \\[-3mm]
$10^5$ & $1/(1.1\times10^5)$ & 178,621\ \ \  & $\begin{array}{r} 10 \\ 100 \\ 200 \end{array}$ & $\begin{array}{r} 10001.7 \\ 1001.1 \\ 500.9  \end{array}$ & 1080 \ \ \ \ \ \ \ \ \ \\[1mm] \hline
\hline \end{tabular}
}
\caption{\sf Example~1---Normal distribution---The setting for each sampled data set, and the resulting values of $\gamma$ from solving \eqref{trap1}--\eqref{trap2}.}
\label{table:normal}
\end{table}

We have set the reference (or desired) distribution function to zero, that is, we have set $w(t) = 0$, for all $t\in[0,1]$.
As can be seen in Figure~\ref{fig:normal}, the estimated density function approximates the true density function better with more data sampled from the distribution, that is, with a larger $n$, as expected.  

\begin{figure}[t!]
\begin{minipage}{80mm}
\begin{center}
\includegraphics[width=80mm]{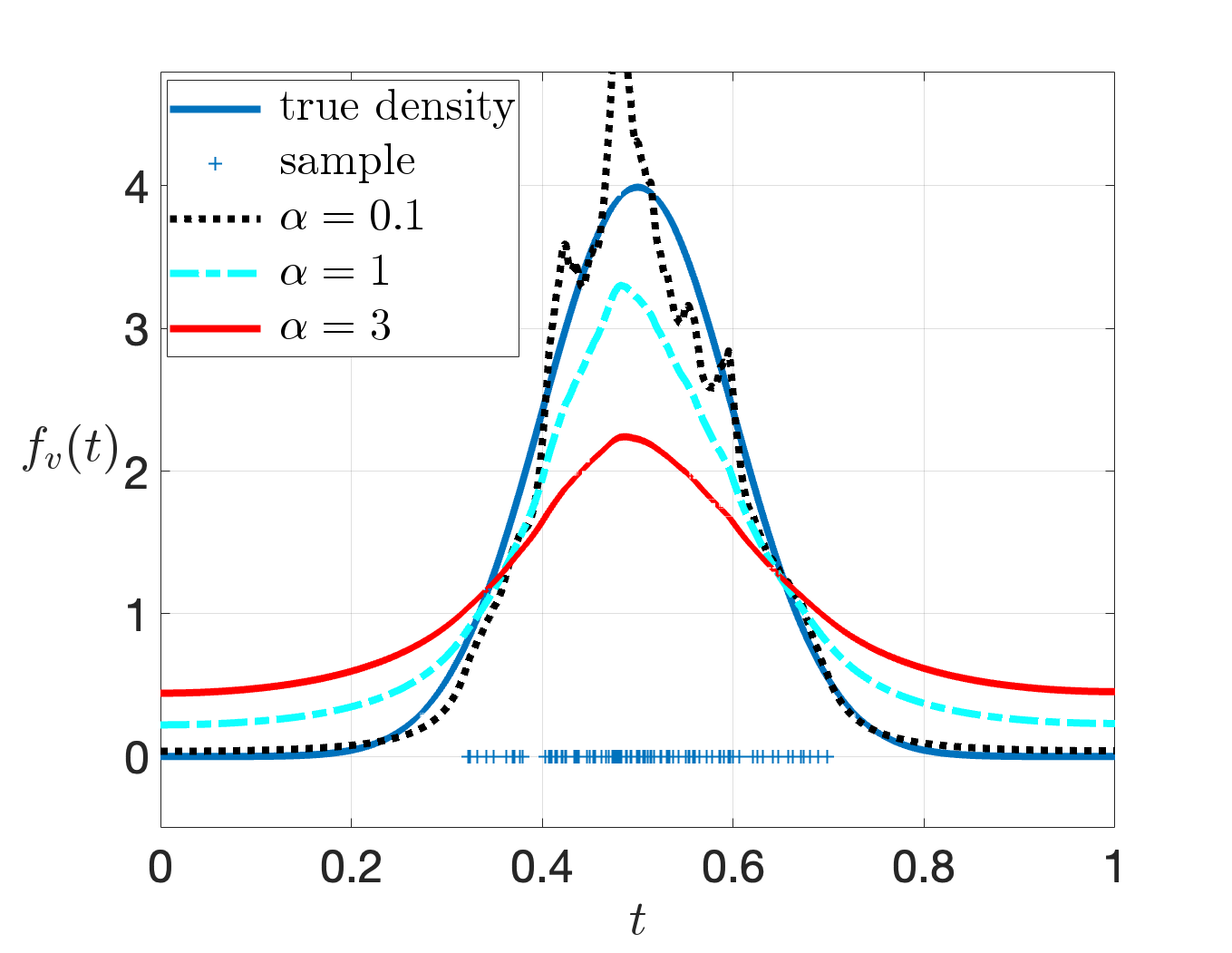}
\\[3mm] (a) 100 sample points
\end{center}
\end{minipage}
\begin{minipage}{80mm}
\begin{center}
\includegraphics[width=80mm]{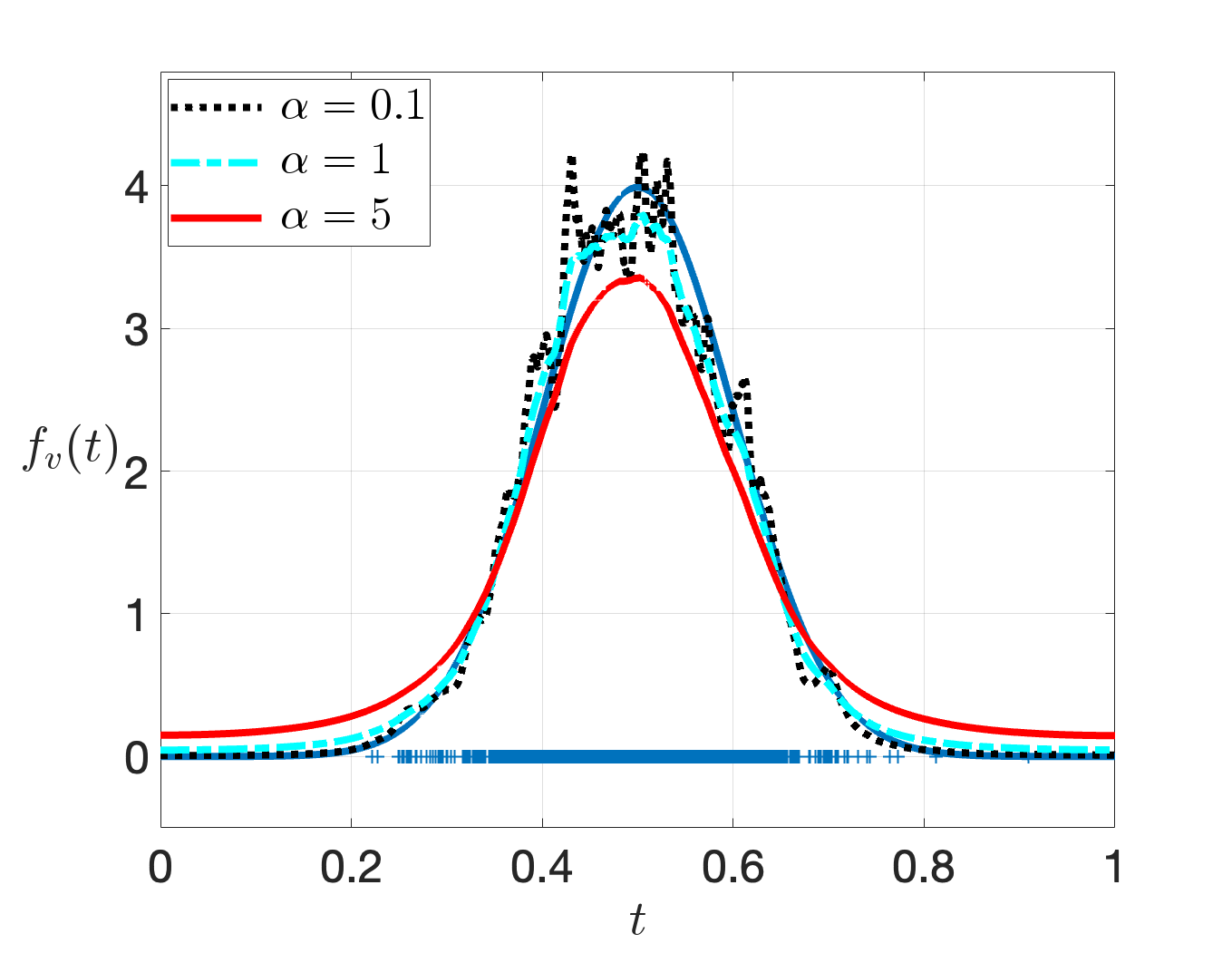} \\[3mm]
(b) 1,000 sample points
\end{center}
\end{minipage} \\[3mm]
\begin{minipage}{80mm}
\begin{center}
\includegraphics[width=80mm]{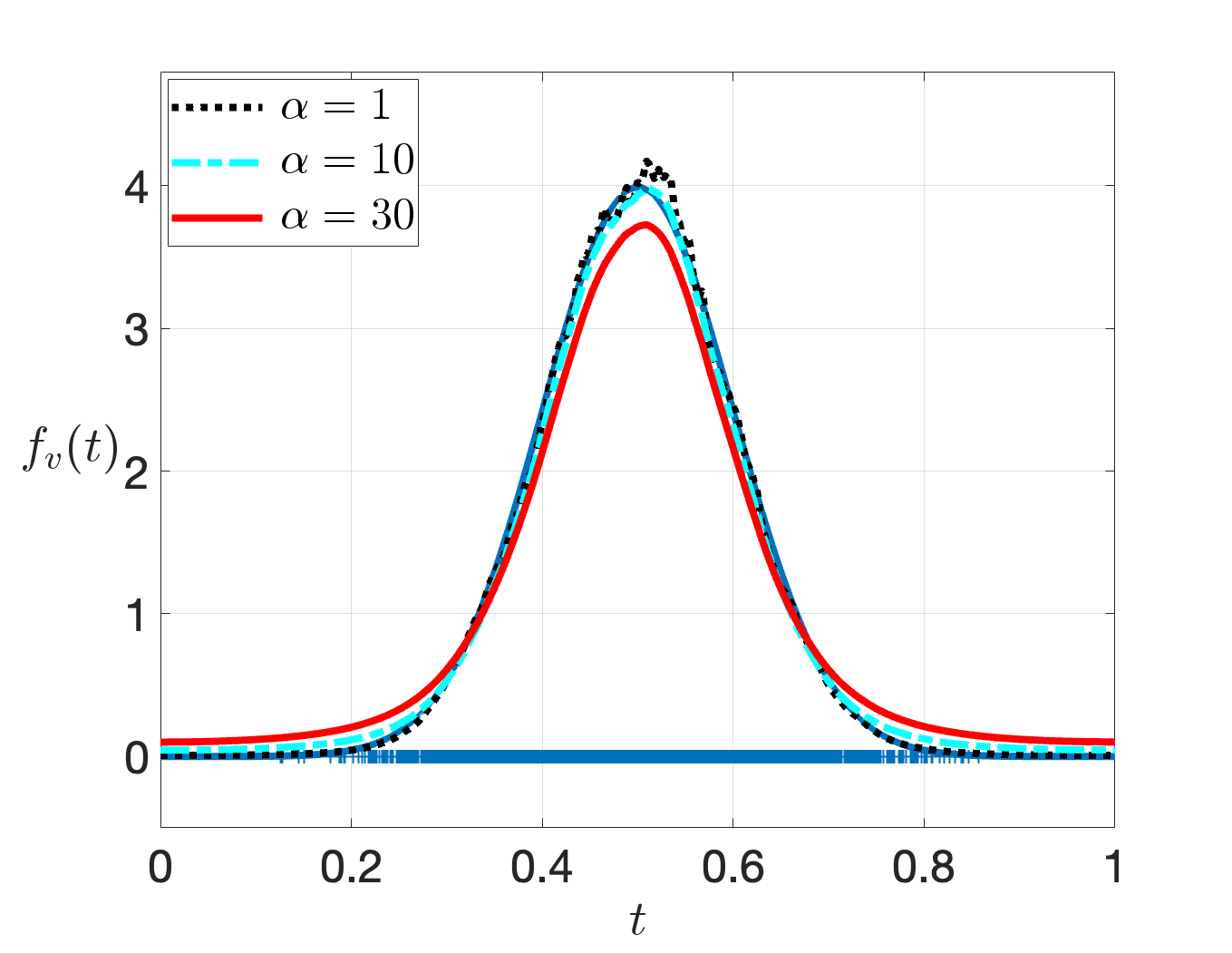} \\[3mm]
(c) 10,000 sample points
\end{center}
\end{minipage}
\begin{minipage}{80mm}
\begin{center}
\includegraphics[width=80mm]{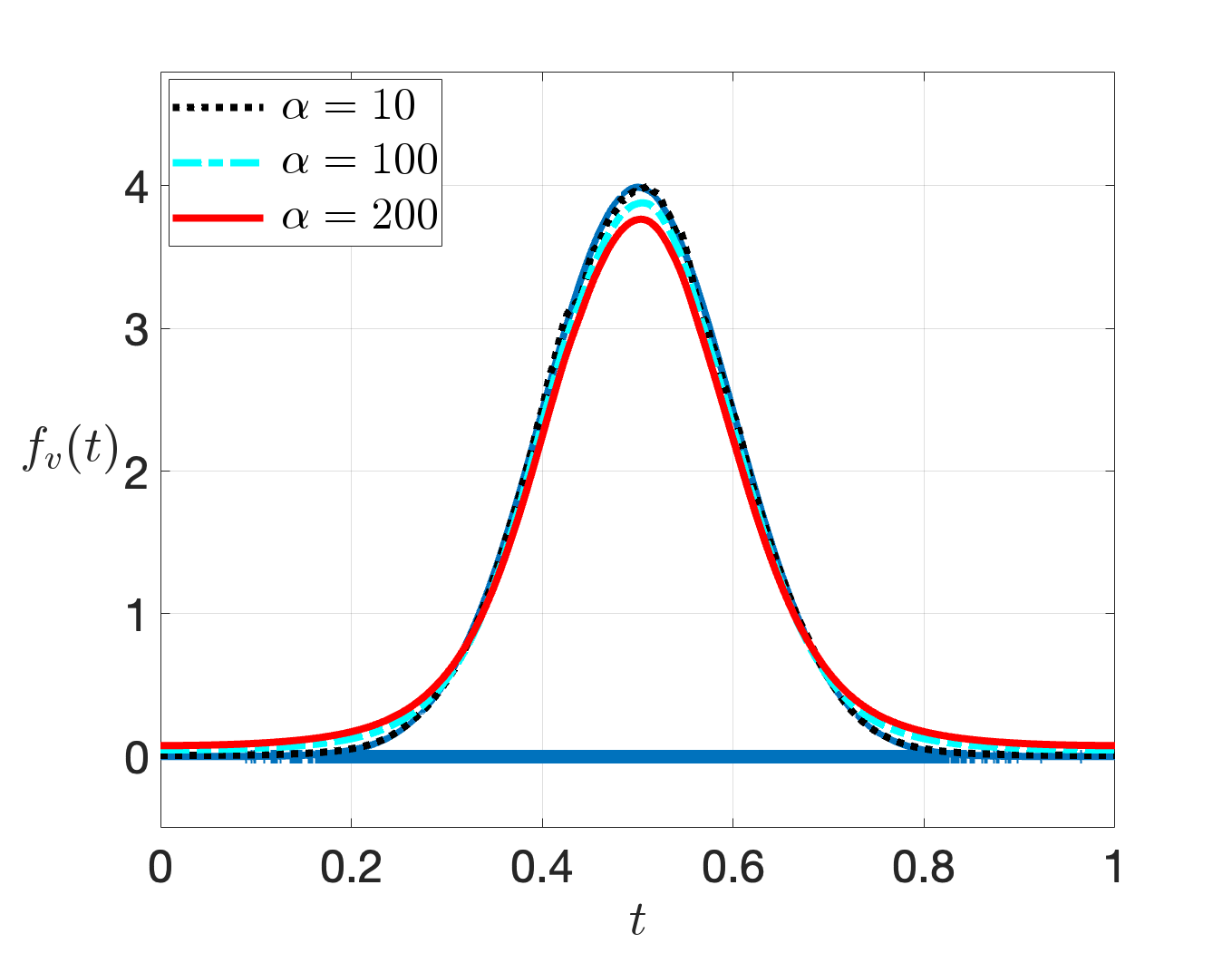} \\[3mm]
(d) 100,000 sample points
\end{center}
\end{minipage}
\caption{\sf Example~1---Normal distribution---density function estimated from various sets of sampled data via optimal control.}
\label{fig:normal}
\end{figure}

The ``smoothness'' of the estimated density (not in the mathematical sense but more in the visual sense) can also be adjusted by varying the value of $\alpha$.  It seems that visual smoothness can be improved at the expense of the accuracy of the estimation, which is expected given the competitive nature of the maximum likelihood and smoothness terms in Problem~(P).  For example, in Figure~\ref{fig:normal}(b), we observe that while the function obtained with $\alpha = 5$ is visually smoother, the function obtained with $\alpha = 1$, albeit more rugged, is closer to the true function (especially at the tail ends).  A similar feature is observed in Figure~\ref{fig:normal}(c), with $\alpha = 10$ and $\alpha = 30$, respectively.  Of course, the choice of an appropriate estimate needs to be left to the practitioner.

Table~\ref{table:normal} provides information on the computational aspects of the density estimation procedure that we employ.  From the CPU times listed, the exponential complexity of the procedure (with increasing values of $n$ and so of $L$) is evident.  It is interesting to observe that the computed value of the constant $\gamma$ is approximately $n/\alpha$, which confirms the arguments made in Remark~\ref{rem:gamma}.

With the same data sets used to obtain the density function estimates depicted in~Figure~\ref{fig:normal}, we have employed the statistical software~R to also estimate the density function; see Figure~\ref{fig:normal_R}.  It is interesting to note that the bandwidth (denoted bw) plays, at least to some extent, a role similar to that of $\alpha$.  We observe that the density functions estimated by the optimal control approach in this paper are of comparable quality to those obtained by the popular statistical software~R.

\afterpage{\clearpage}
\begin{figure}
\begin{minipage}{80mm}
\begin{center}
\includegraphics[width=80mm]{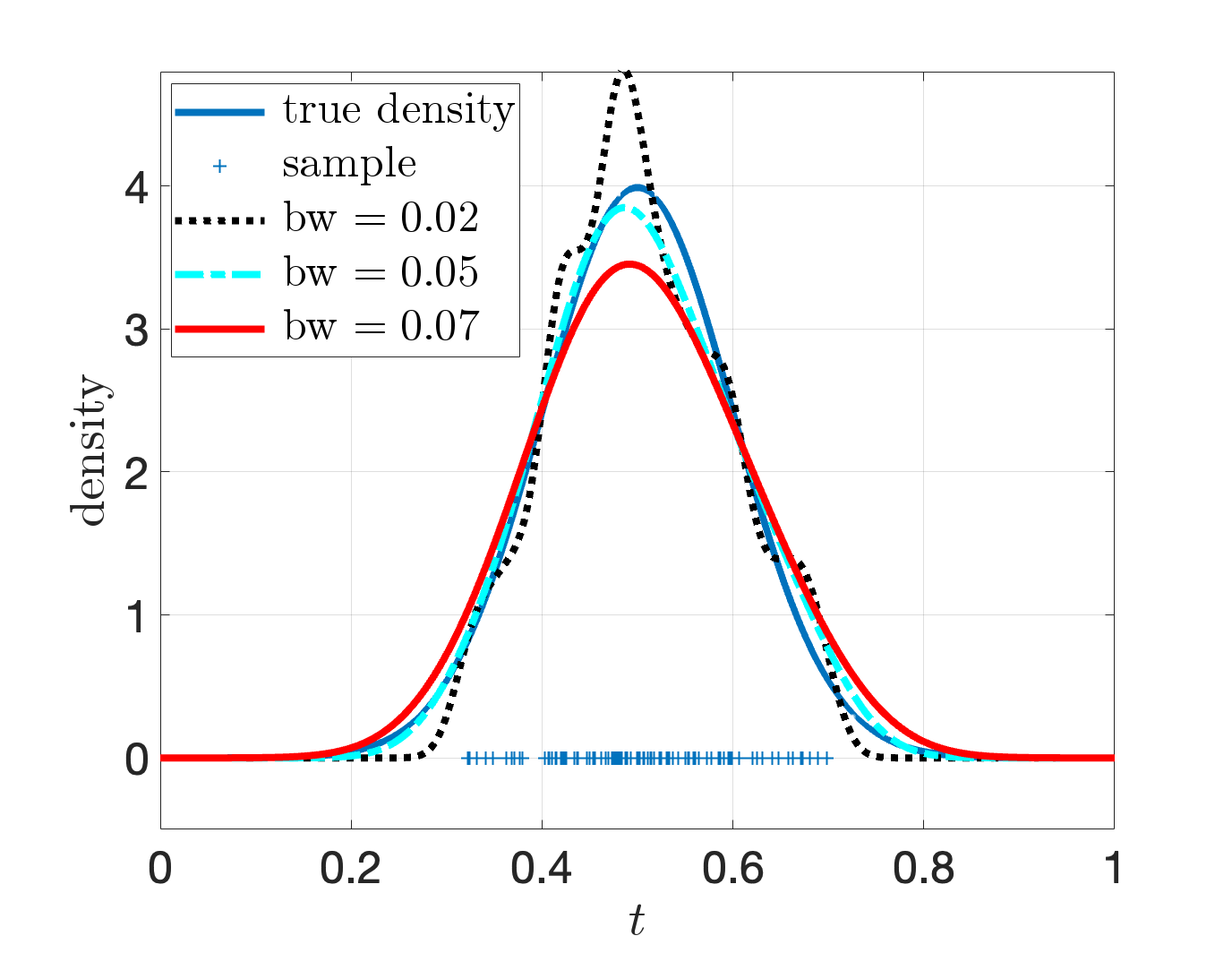}
\\[3mm] (a) 100 sample points
\end{center}
\end{minipage}
\begin{minipage}{80mm}
\begin{center}
\includegraphics[width=80mm]{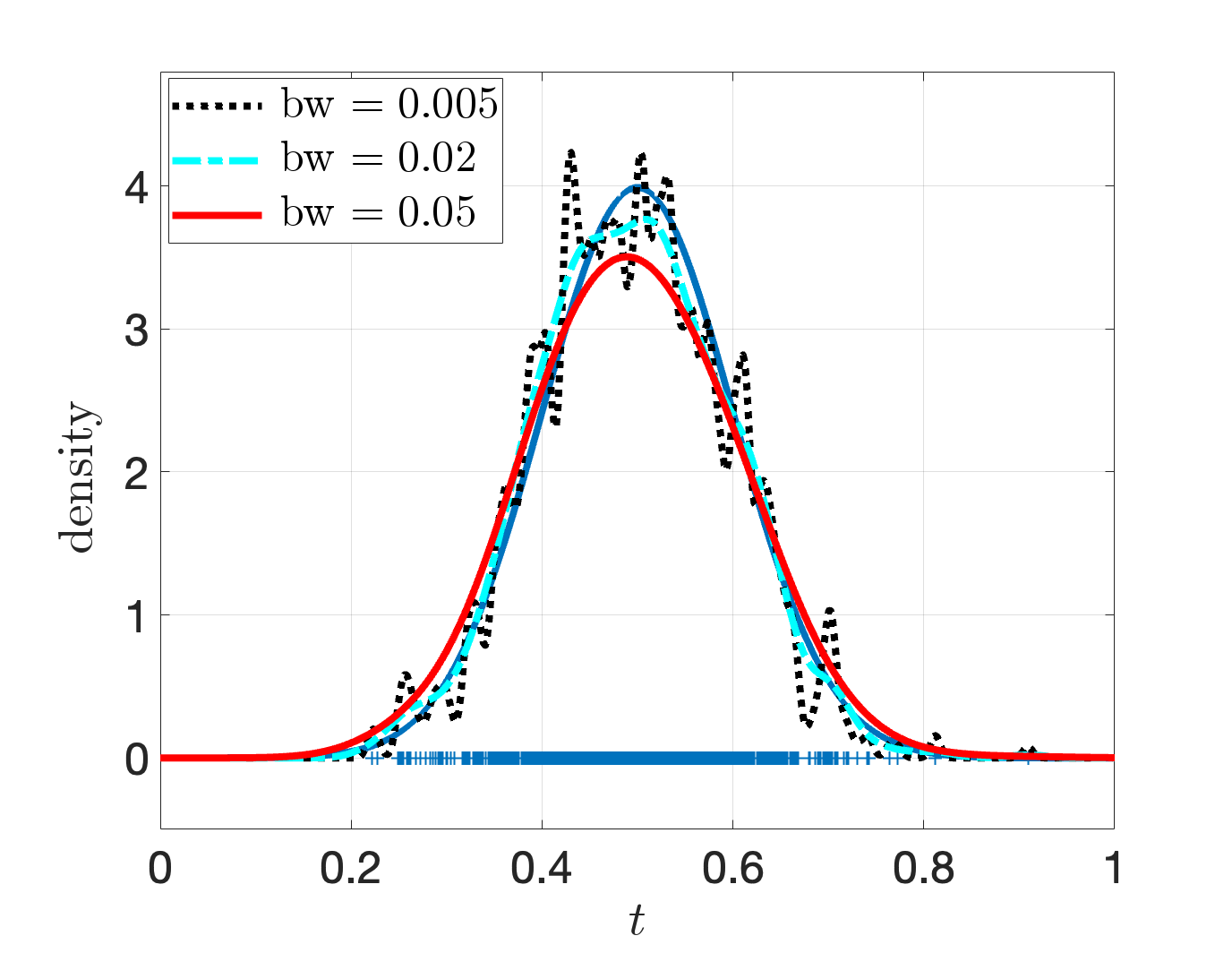} \\[3mm]
(b) 1,000 sample points
\end{center}
\end{minipage} \\[3mm]
\begin{minipage}{80mm}
\begin{center}
\includegraphics[width=80mm]{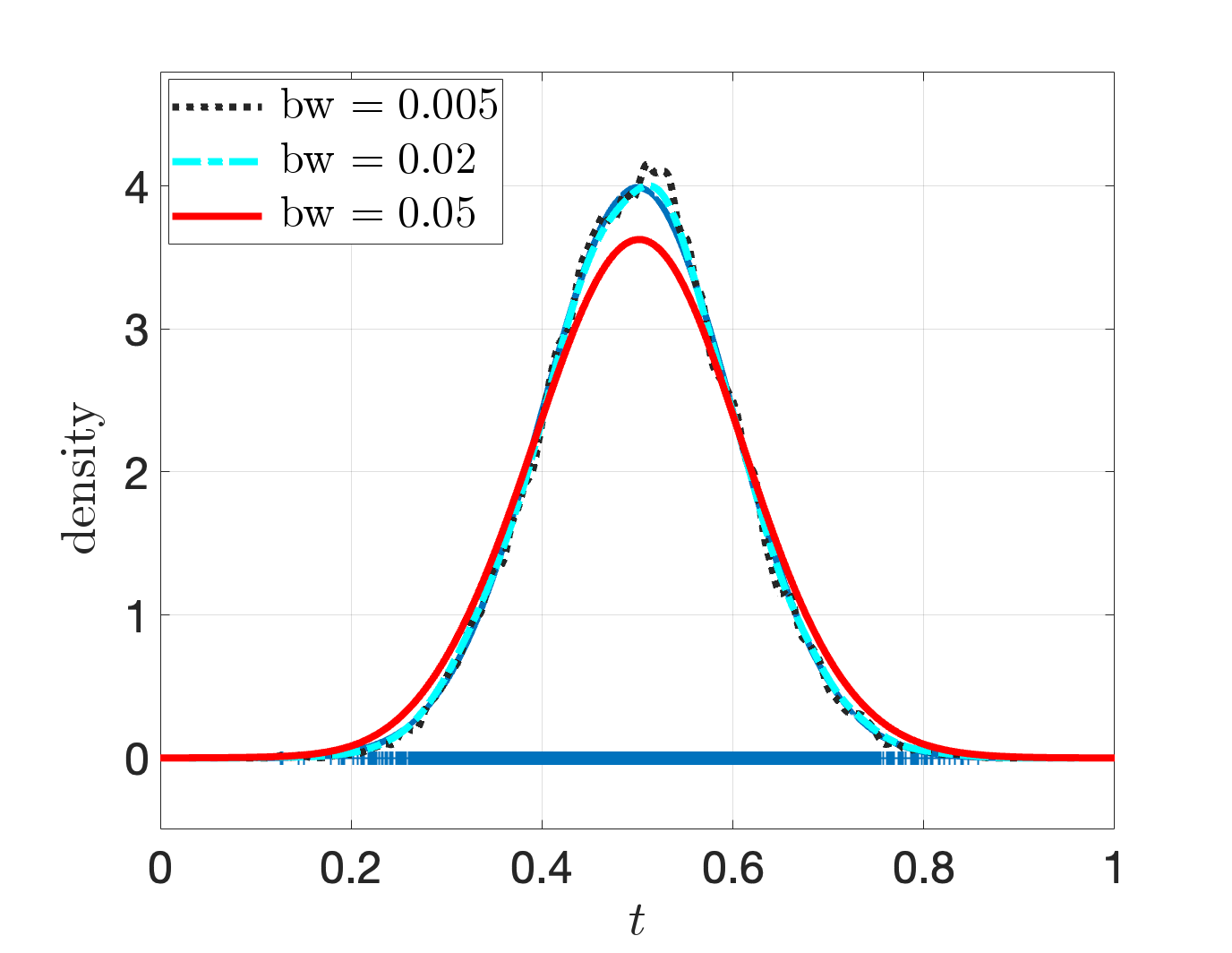} \\[3mm]
(c) 10,000 sample points
\end{center}
\end{minipage}
\begin{minipage}{80mm}
\begin{center}
\includegraphics[width=80mm]{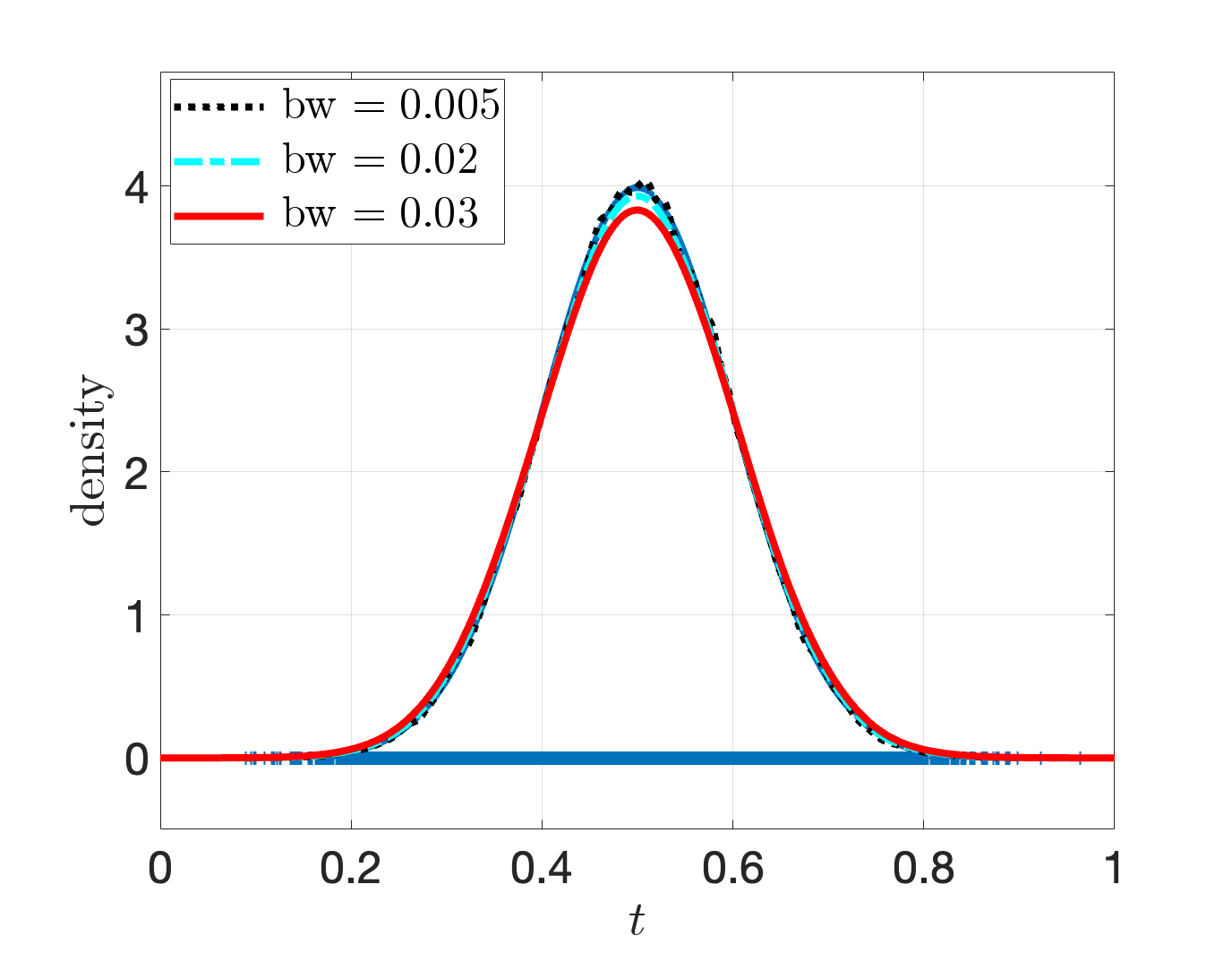} \\[3mm]
(d) 100,000 sample points
\end{center}
\end{minipage}
\caption{\sf Example~1---Normal distribution---density function estimated from various sets of sampled data using the kernel method in R, using the same data sets as in Figure~\ref{fig:normal}.}
\label{fig:normal_R}
\end{figure}

\newpage
\subsubsection{Example 2: The Old Faithful}

In Example~1, we used ``synthetic data sets''.  This time, we consider a real-life data set that contains 272 observations of the durations of eruptions, measured in minutes, of the Old Faithful geyser in Yellowstone National Park, Wyoming, the United States~\cite{AzzBow1990}.  This data set has been widely used to test the performance of density estimators~\cite{GriHeg2010, Kwasniok2021, FusHorTsu2006}; also see~\cite[Chapter~8]{HotEve2014} for the Old Faithful geyser observations of the waiting times between consecutive eruptions.

Figure~\ref{fig:OF}(a) shows the estimates of the density function of different degrees of smoothness for various values of $\alpha$ and $\beta=1$, obtained by solving \eqref{trap1}--\eqref{trap2}.  We have set $w(t) = 0$, for all $t\in[0,1]$.  Table~\ref{table:OF} lists the computational aspects of the density estimation procedure we use, as well as the optimal values of $\gamma$ corresponding to various values of $\alpha$.  As in Example~1, the ``smoothness'' in the appearance of the density function can be improved by increasing the value of $\alpha$ here.

In Figure~\ref{fig:OF}(b), we provide a histogram of the data given, which serves as a valuable reference for comparisons.  Figure~\ref{fig:OF}(c), on the other hand, shows the density functions estimated by R for various bandwidths.  Although the density function estimated by R, for example, for $\mbox{bw} = 0.3$, shows a slightly smoother appearance (by the very nature of the kernel methods), the density function in~Figure~\ref{fig:OF}(a), estimated by the optimal control approach for $\alpha = 1$, seems to represent the histogram more closely.

\begin{table}[t!]
\centering
{\small
\begin{tabular}{cccccc} \hline \\[-2mm]
$n$ & $h$ & $L$ & $\alpha$ & $\gamma$ & \hspace*{-3mm}$\begin{array}{c} \mbox{CPU time} \\ \mbox{[sec]} \end{array}$ \\[4mm] \hline \\[-3mm]
$272$ & $1/(2\times10^3)$ & 2,198 & $\begin{array}{r} 0.1 \\ 0.5 \\ 1\ \ \\ 2\ \ \end{array}$ & $\begin{array}{r} 13611.9 \\ 2730.0 \\ 1369.3 \\ 688.7 \end{array}$ & 0.14 \\[1mm] 
\hline\hline \end{tabular}
}
\caption{\sf Example~2---The Old Faithful---The setting for the Old Faithful data set and the resulting values of $\gamma$ from solving \eqref{trap1}--\eqref{trap2}.}
\label{table:OF}
\end{table}

\afterpage{\clearpage}
\begin{figure}[t!]
\begin{minipage}{80mm}
\begin{center}
\includegraphics[width=80mm]{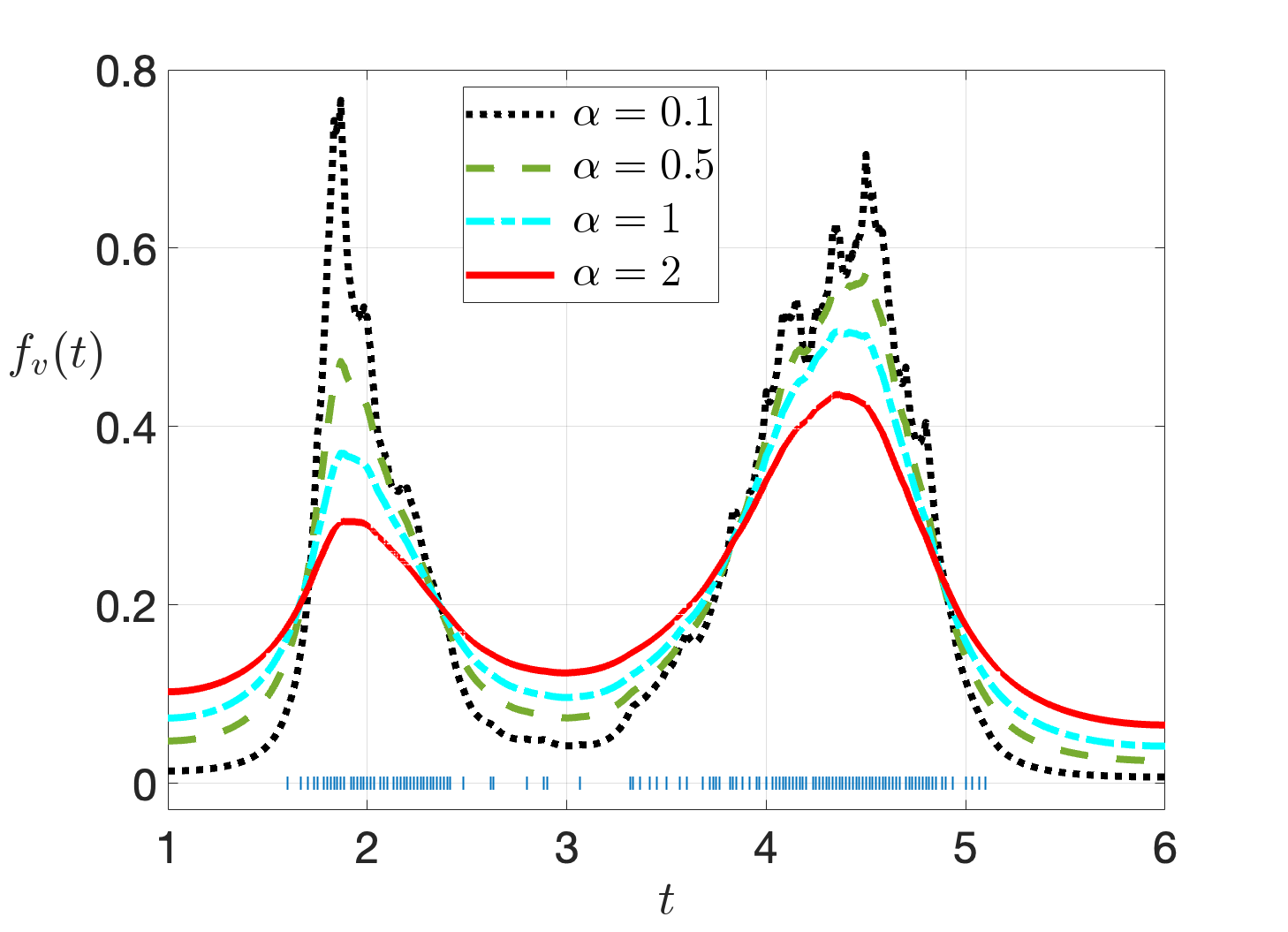}
\\[3mm] (a) Density via optimal control
\end{center}
\end{minipage}
\begin{minipage}{80mm}
\begin{center}
\includegraphics[width=80mm]{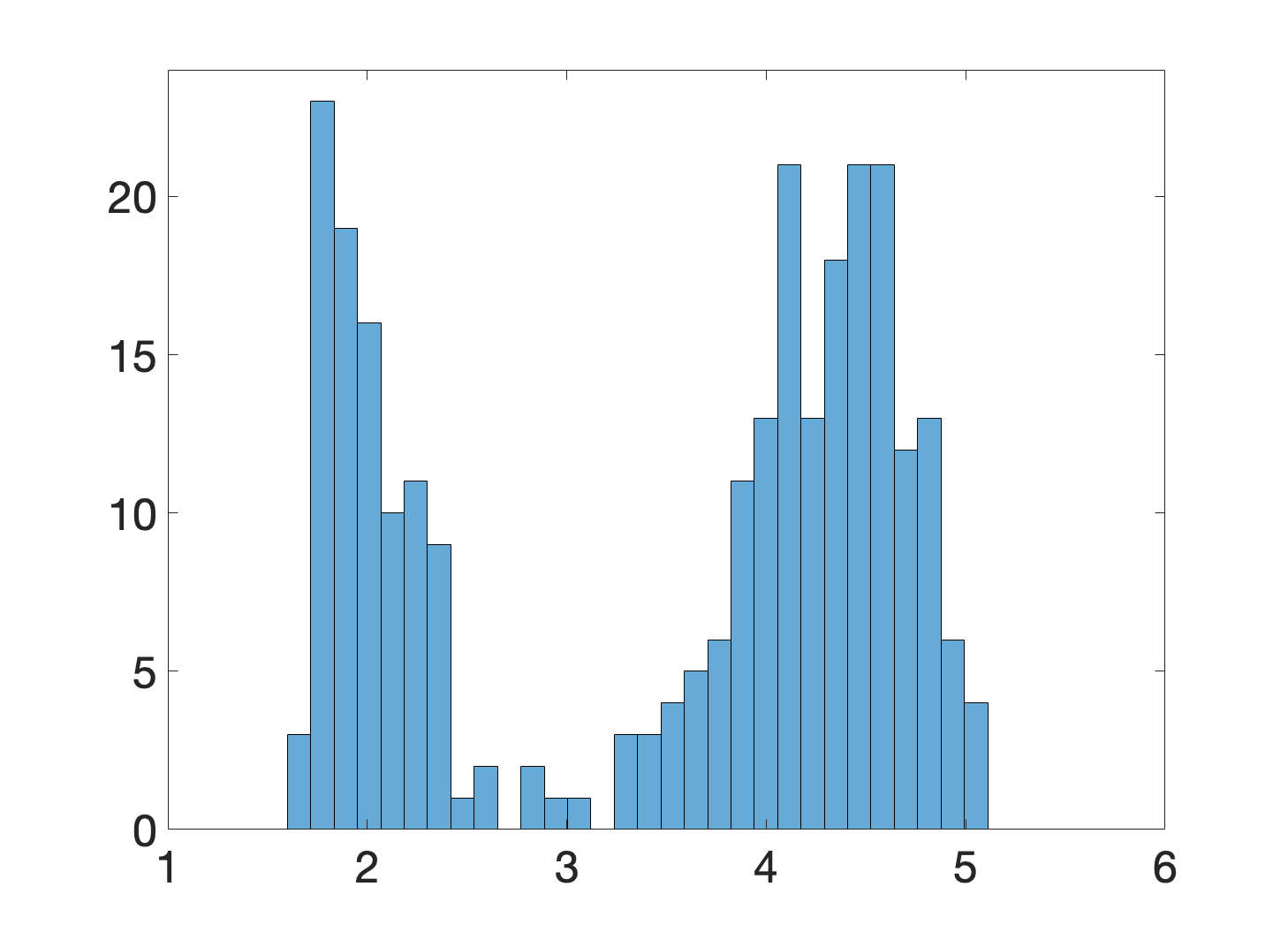} \\[3mm]
(b) Histogram
\end{center}
\end{minipage} \\[3mm]
\begin{minipage}{80mm}
\begin{center}
\includegraphics[width=80mm]{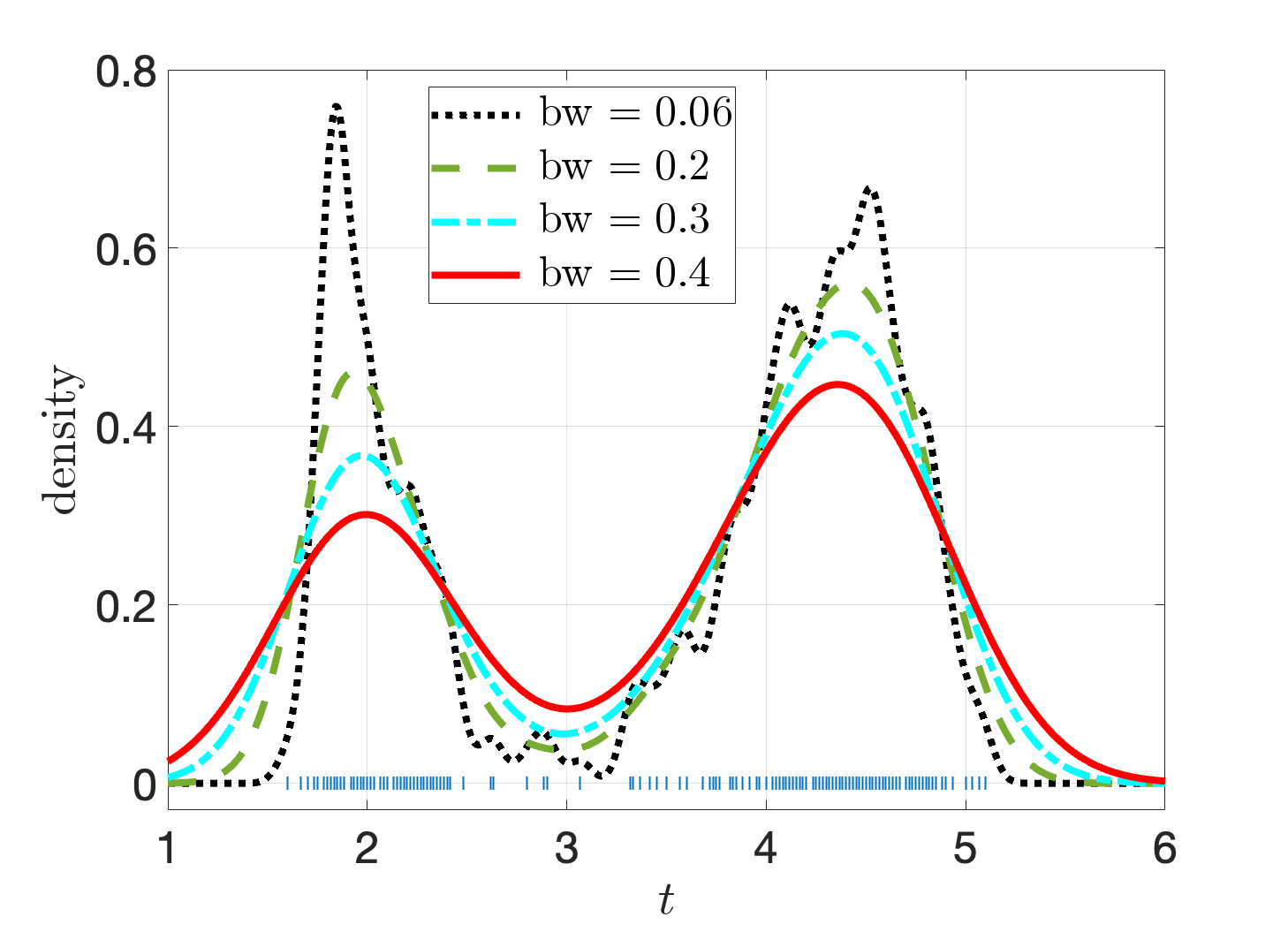} \\[3mm]
(c) Density via the kernel method in R
\end{center}
\end{minipage}
\begin{minipage}{80mm}
\end{minipage}
\caption{\sf Example~2---The Old Faithful---estimated density functions.}
\label{fig:OF}
\end{figure}

\newpage
\subsubsection{Example 3: Galaxy speeds}

Another real-life data set used to test density function estimators is the set of ``heliocentric speeds'', measured in kilometres per second, of 83 galaxies in the Corona Borealis region~\cite[Table~1]{PosHucGel1986}.  Here, we will show the speeds in the graphs in thousands of kilometres per second, for neatness of exposition.
The first attempt to estimate the density function for these data of galaxy speeds seems to have been made in~\cite[Table~1]{Roeder1990}, albeit by using the speeds of 82 of the galaxies, instead of all 83 of them, leaving out the one with the speed 5,607 km/s.  As far as we can judge, in all subsequent studies in which density function estimators have been tested using galaxy speeds (see, for example, \cite[Chapter~8]{HotEve2014} and \cite{FusHorTsu2006}), the data set of 82 galaxy speeds as given in~\cite{Roeder1990} has been used. Here, we include all 83 galaxies as given in~\cite{PosHucGel1986}.

Figure~\ref{fig:GS}(a) shows the density functions of different degrees of smoothness for various values of $\alpha$ and $\beta=1$, obtained by solving \eqref{trap1}--\eqref{trap2}.  As in Example~2, we have set $w(t) = 0$, for all $t\in[0,1]$.  In Table~\ref{table:GS}, the computational aspects of the procedure are listed, as well as the optimal values of $\gamma$ corresponding to various values of $\alpha$, as in previous examples.

In Figure~\ref{fig:GS}(b), a histogram of the data is provided, and  Figure~\ref{fig:GS}(c) shows the density functions estimated by R for various bandwidths.  Although the density function estimated by R, for example, for $\mbox{bw} = 3$, shows a smoother appearance, the density function in~Figure~\ref{fig:GS}(a), estimated by the optimal control approach for $\alpha = 1$, seems to represent the histogram more closely, especially in the middle part of the distribution.

\begin{table}[t!]
\centering
{\small
\begin{tabular}{cccccc} \hline \\[-2mm]
$n$ & $h$ & $L$ & $\alpha$ & $\gamma$ & \hspace*{-3mm}$\begin{array}{c} \mbox{CPU time} \\ \mbox{[sec]} \end{array}$ \\[4mm] \hline \\[-3mm]
$83$ & $1/(2\times10^3)$ & 2,043 & $\begin{array}{r} 0.1 \\ 0.5 \\ 1\ \ \\ 2\ \ \end{array}$ & $\begin{array}{r} 33382.6 \\ 6807.1 \\ 3480.0 \\ 1814.1 \end{array}$ & 0.15 \\[1mm] 
\hline\hline \end{tabular}
}
\caption{\sf Example~3---Galaxies---The setting for the Galaxy speeds data set and the resulting values of $\gamma$ from solving \eqref{trap1}--\eqref{trap2}.}
\label{table:GS}
\end{table}

\afterpage{\clearpage}
\begin{figure}[t!]
\begin{minipage}{80mm}
\begin{center}
\includegraphics[width=80mm]{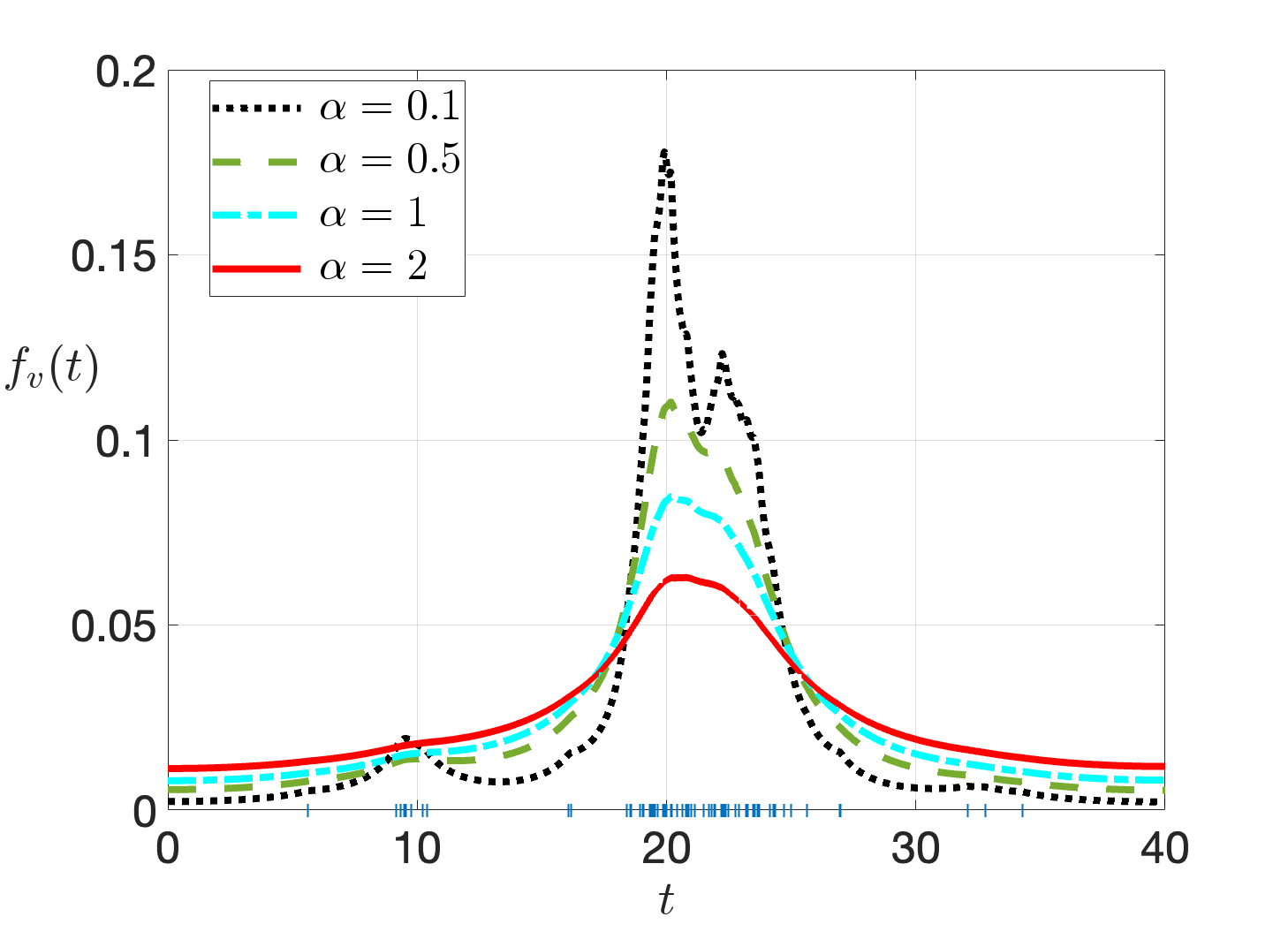}
\\[3mm] (a) Density via optimal control
\end{center}
\end{minipage}
\begin{minipage}{80mm}
\begin{center}
\includegraphics[width=80mm]{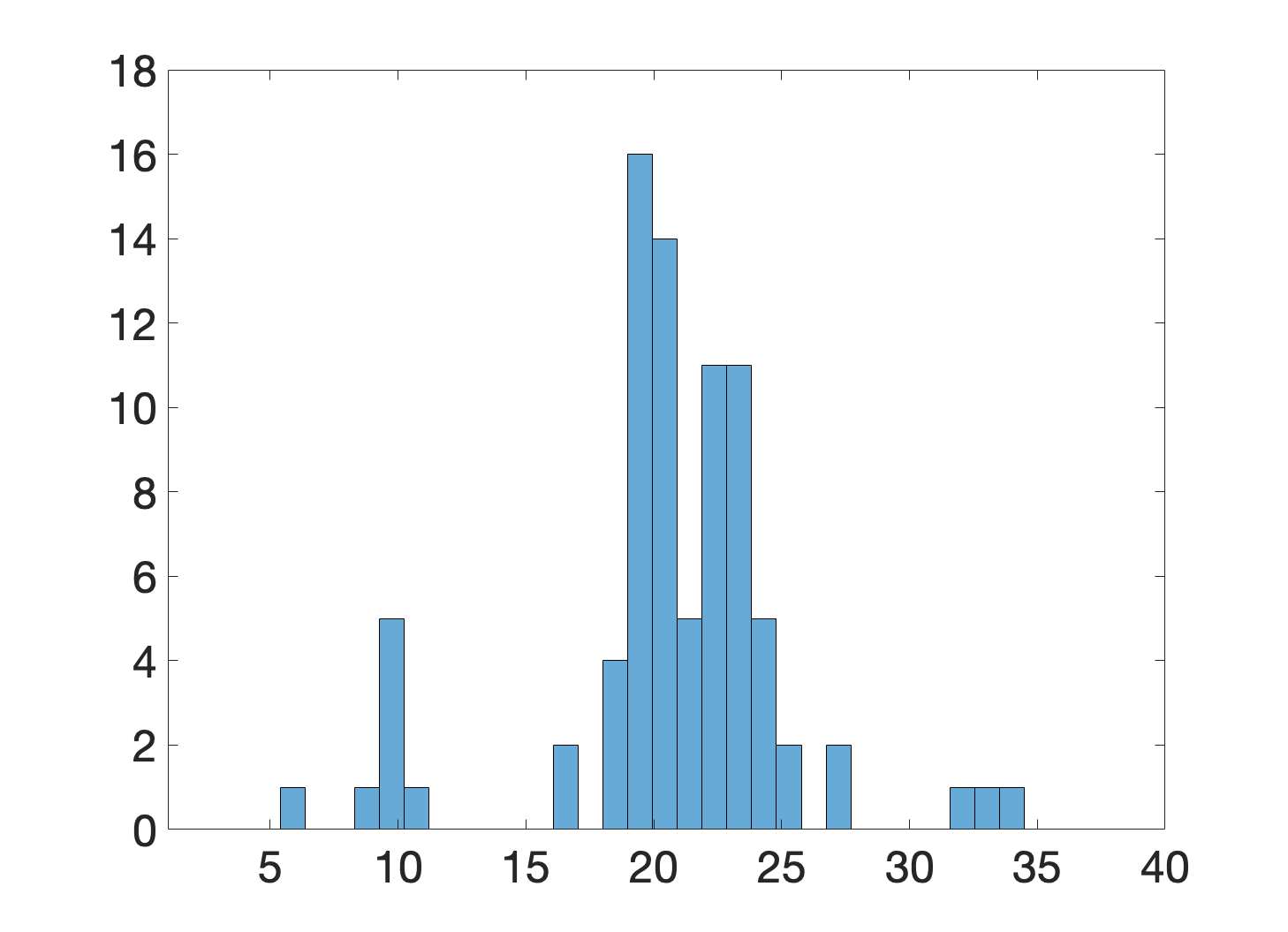} \\[3mm]
(b) Histogram
\end{center}
\end{minipage} \\[3mm]
\begin{minipage}{80mm}
\begin{center}
\includegraphics[width=80mm]{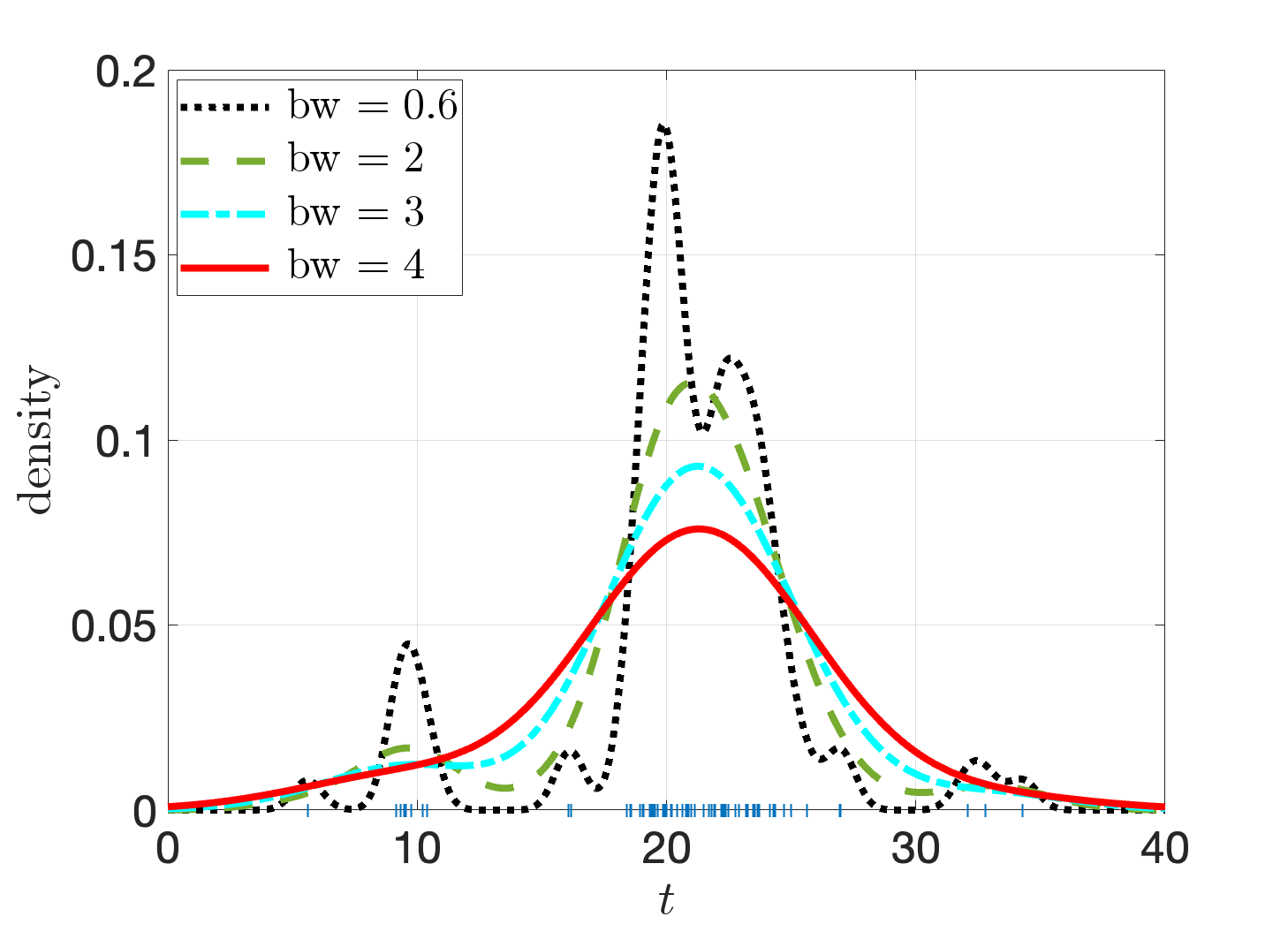} \\[3mm]
(c) Density via the kernel method in R
\end{center}
\end{minipage}
\begin{minipage}{80mm}
\end{minipage}
\caption{\sf Example~3---Galaxies---estimated density functions.}
\label{fig:GS}
\end{figure}

\section{Conclusion}
\label{conclusion}
We have reformulated the density estimation problem as a multi-stage optimal control problem.  We proposed a new numerical approach to solve the two-point boundary-value problem emanating from the maximum principle and obtain an estimate of the probability density function.  We demonstrated the working of the numerical method on example data sets.

It would be interesting to consider Problem~(P) with a more general $f$ rather than a specific $f(t) = e^{v(t)}$.  It would also be interesting to impose simple bounds on the ``control'' $\dot{f}(t)$, as was done without the additional regularization terms in~\cite{Shvartsman2010}.  Adding additional constraints involving $f$ would make Problem~(P) more challenging to tackle both theoretically and numerically.

It would also be interesting to consider additional constraints such as those on moments, quantiles, and entropy, as discussed in~\cite{HalPre1999}.  Optimal control theory and computations are particularly well-known to handle constraints well compared to classical calculus of variations formulations.

\section*{Acknowledgments}
The second-named author is grateful to John Hinde for pointing out the data set of galaxy speeds for density estimation.


\end{document}